\documentclass[]{article}

\usepackage{subcaption}
\usepackage{mathtools}
\usepackage{amsmath, amsfonts, amsthm}
\usepackage{url}
\usepackage{longtable}
\usepackage{appendix}

\newtheorem{thm}{Theorem}
\newtheorem{lem}{Lemma}

\title{An Extension of Hejhal's Algorithm to Infinite Volume Fundamental Domains}
\author{Alex Karlovitz\footnote{Partially supported by NSF grant DMS-1802119 of Alex Kontorovich}}
\date{}

\begin{document}

\maketitle

\begin{abstract}
This work presents an algorithm for numerically computing Maass forms and their eigenvalues for Fuchsian groups of infinite covolume.
By Patterson-Sullivan theory, this has the added benefit of computing Hausdorff dimensions of the limit sets of these groups.
To approximate Maass forms, we consider their Fourier expansions in different coordinate systems.
To handle infinite volume fundamental domains, we make use of the concept of flare domains.
We also develop theory about Fourier expansions in flare domains which mimics the classical theory on expansions with respect to cusps.
Finally, we present detailed examples of the algorithm applied to symmetric Schottky groups and infinite volume Hecke groups.
\end{abstract}


\section{Introduction}\label{ch:Intro}

\begin{figure}[h!]
	\centering
	\begin{subfigure}{.5\textwidth}
		\centering
		
		\begin{tabular}{|c|c|}
			\hline
			$a_0$ & 1.0 \\
			\hline
			$a_1$ & 1.3915582132 \\
			\hline
			$a_2$ & 1.042527677 \\
			\hline
			$a_3$ & 0.4185419039 \\
			\hline
			$a_4$ & 0.0082837207 \\
			\hline
			$b_0$ & 0.5693837173 \\
			\hline
			$b_1$ & -8.6593303021E-5 \\
			\hline
			$b_2$ & 2.7659692148E-10 \\
			\hline
			$b_3$ & -2.1134025239E-14 \\
			\hline
			$b_4$ & 4.6505587029E-19 \\
			\hline
		\end{tabular}
		\caption{First five cuspidal coefficients ($a_n$'s) and flare coefficients ($b_n$'s) to 10 digits of precision.}
		\label{efunc:table}
	\end{subfigure}%
	\begin{subfigure}{.5\textwidth}
		\centering
		\includegraphics[width=.9\linewidth]{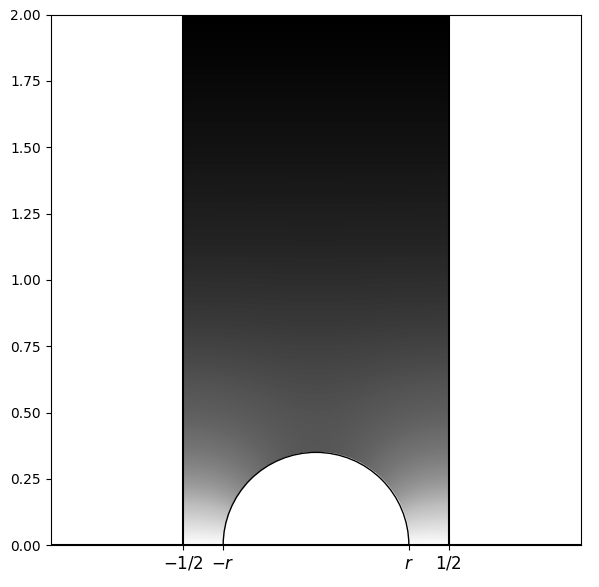}
		\caption{Heat map on a fundamental domain.}
		\label{efunc:fd}
	\end{subfigure}
	\caption{The base eigenfunction for the infinite volume Hecke triangle group with parameter $r = 0.35$.}
	\label{efuncHecke}
\end{figure}

In this paper, we are interested in the spectral analysis of $L^2$ functions on hyperbolic quotient spaces.
Specifically, we consider a Fuchsian group $\Gamma$ acting on the hyperbolic plane $\mathbb{H}$ and attempt to study the associated Maass forms on $L^2(\Gamma\backslash\mathbb{H})$.
Hejhal's algorithm \cite{hejhal1992eigenvalues} allows one to numerically compute the eigenvalues and Fourier coefficients of these automorphic forms.
In its original design, the algorithm has two key requirements for the Fuchsian group:
\begin{enumerate}
	\item $\Gamma$ must have finite covolume
	\item $\Gamma$ must contain a full rank parabolic subgroup
\end{enumerate}
In this paper, we remove both of these requirements.
Instead, we require only that the group be Zariski dense and geometrically finite and that the Hausdorff dimension $\delta$ of the limit set of $\Gamma$ exceeds $1/2$.
(See Appendix \ref{sec:H3} for a discussion of the case of higher dimensional hyperbolic space).

In the infinite covolume case, this work has geometric applications.
According to Patterson-Sullivan theory \cite{patterson1976limit, sullivan1984entropy}, the smallest eigenvalue $\lambda_0$ of the Laplacian on $L^2(\Gamma\backslash\mathbb{H})$ is related to $\delta$ by the formula
$$
\lambda_0 = \delta(1 - \delta)
$$
Since our proposed algorithm efficiently approximates the eigenvalues of the Laplacian, we can use the above formula to also obtain sharp estimates for Hausdorff dimensions of limit sets.

An alternative approach was studied by McMullen, who proposed an algorithm utilizing Markov partitions to compute the Hausdorff dimension of limit sets of Kleinian groups and Julia sets of rational maps \cite{mcmullen1998hausdorff}.
By the above equation, this also effectively computes the smallest eigenvalue of the Laplacian.
McMullen applied his algorithm to infinite volume Hecke groups and symmetric Schottky groups, which are the two main examples discussed in this paper.
The technique used here - making use of Hejhal's algorithm - is able to compute more decimal places and runs somewhat more quickly.
Moreover, Hejhal's algorithm estimates the Fourier coefficients of the Maass form, meaning we obtain an estimate for the \textit{eigenfunction} in addition to the estimated eigenvalue and Hausdorff dimension.
Even further, the algorithm here can potentially be used to find other eigenfunctions and eigenvalues (in the case where any exist \cite{phillips1985spectrum}).

In Figure \ref{efuncHecke}, we show an example of a base eigenfunction computed using the new algorithm (this example is discussed in more detail in Section \ref{ch:Hecke}).
In Section \ref{sec:BaseEfuncs}, we will see that base eigenfunctions are everywhere positive, so the heat map shown in Figure \ref{efunc:fd} is an appropriate tool for visualizing the Maass form.
Note in particular that the function goes to 0 in the infinite volume portion of the fundamental domain.
The eigenfunction here is not a cusp form, and we have normalized it to approach 1 at the cusp.
To get a more detailed view of the heat map, we look at a zoomed-in version in Figure \ref{efunc_zoomed}.
Looking closely, one can spot what appear to be level curves of the eigenfunction running along geodesics perpendicular to the boundaries of the fundamental domain.

\begin{figure}
	\centering
	\includegraphics[width=.9\linewidth]{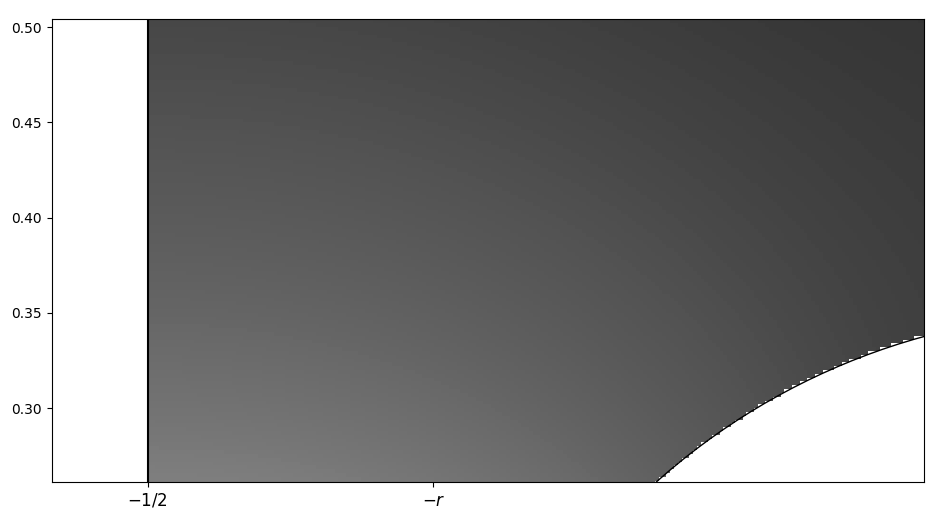}
	\caption{A zoomed-in portion of Figure \ref{efunc:fd}.}
	\label{efunc_zoomed}
\end{figure}

The code used for producing all results and figures in this paper can be found on GitHub at \url{https://github.com/alexkarlovitz/hejhal}.
As seen there, some of the functions were borrowed from Andreas Str\"ombergsson.
In fact, this work was founded on unpublished code and notes by Kontorovich and Str\"ombergsson $\cite{kontorovichUP}$.
See Appendix \ref{ch:Appx_Code} for detailed notes on using the GitHub code.

\subsection{Flare Domains}\label{sec:IntroFlares}

After discussing preliminaries in Section \ref{ch:Prelims}, we begin by describing how Hejhal's algorithm can be adapted to the context of \textit{flare domains} in Section \ref{ch:HejhalInFlare}.

Any Zariski dense Fuchsian group for $\text{SL}(2, \mathbb{R})$ contains an infinite number of hyperbolic elements.
Hyperbolic matrices are diagonalizable over $\text{SL}(2, \mathbb{R})$, and thus we may assume (perhaps after conjugation) that such a Fuchsian group $\Gamma$ contains a diagonal matrix, say
\[
A_\kappa =
\begin{pmatrix}
	\sqrt{\kappa} & 0 \\
	0 & 1/\sqrt{\kappa}
\end{pmatrix}
\]
for some $\kappa > 1$.
Note that the action of this matrix on $\mathbb{H}$ is $z \mapsto \kappa z$.
Thus, a fundamental domain for the action on $\mathbb{H}$ of the (elementary) group generated by \textit{this matrix only} is
$$
\{ z \in \mathbb{H} : 1 < |z| < \kappa \},
$$
and so there exists a fundamental domain for the action of all of $\Gamma$ which is a subset of this half-annulus.

Now assume that a Fuchsian group $\Gamma$ has infinite volume and is finitely generated.
Since any fundamental domain for $\Gamma$ has infinite hyperbolic volume, there must be some fundamental domain whose boundary contains an interval of positive measure on the real line (as opposed to simply a cusp).
In Section \ref{sec:FlareDef}, we show an example in which such a fundamental domain can be conjugated to one with a subset of the form
$$
\{ z \in \mathbb{H} : 1 < |z| < \kappa, \arg z < \alpha \}
$$
We refer to this subset (or its preimage before conjugation) as a \textit{flare}, and we call a fundamental domain containing a flare a \textit{flare domain}.
See Figure \ref{flare_example} for an example.
\begin{figure}
	\centering
	\includegraphics[width=\linewidth]{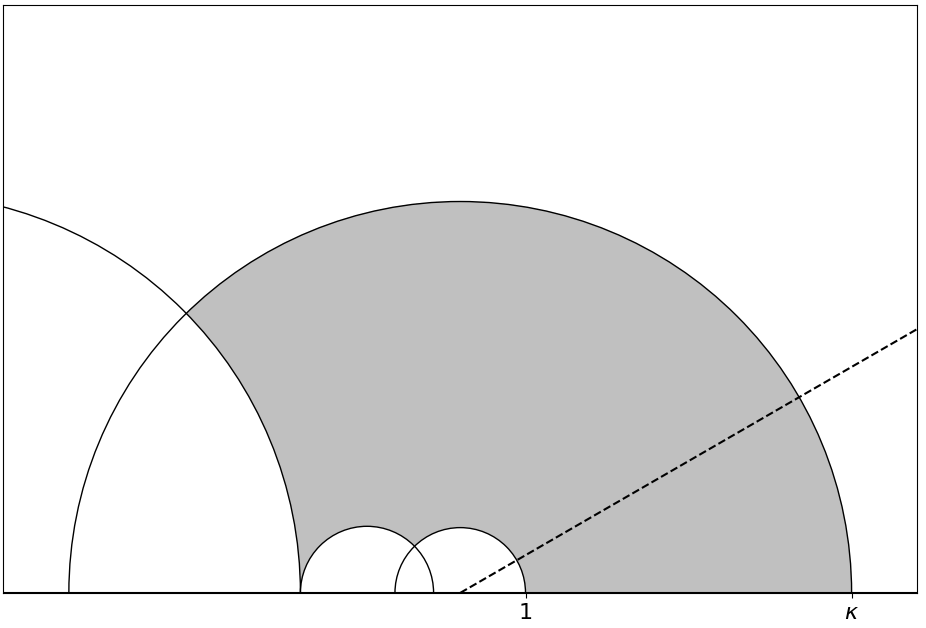}
	\caption{Example flare domain. The dotted line cuts off a flare at angle $\alpha = \pi/6$.}
	\label{flare_example}
\end{figure}
This example is taken from an infinite volume Hecke group, and it is discussed in more detail in Section \ref{ch:Hecke}.

\subsection{Schottky Groups}\label{sec:IntroSchottky}

In Section \ref{ch:Schottky}, we apply our method to some \textit{Schottky groups}, by which we mean groups generated by reflections through a finite set of non-intersecting geodesics in the hyperbolic plane.
In particular, we consider \textit{symmetric} Schottky groups in the disk model of hyperbolic 2-space.
These were studied by McMullen, who gave estimates for the Hausdorff dimension of the limit sets of such groups \cite{mcmullen1998hausdorff}.

Symmetric Schottky groups are a simple example of Schottky groups.
These are obtained by taking 3 circles of the same size spaced symmetrically about the disk.
We parameterize such sets of circles by $\theta$, the angle along the unit circle cut out by one of the circles.
See Figure \ref{pi_over_2} for an example.
\begin{figure}
	\centering
	\begin{subfigure}{.5\textwidth}
		\centering
		\includegraphics[width=.9\linewidth]{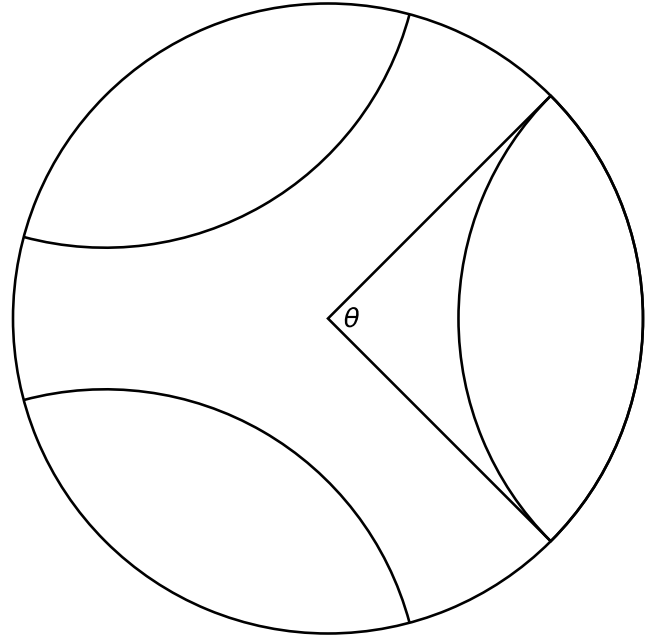}
		\caption{Circles defining the group.}
		\label{schottky:suba}
	\end{subfigure}%
	\begin{subfigure}{.5\textwidth}
		\centering
		\includegraphics[width=.9\linewidth]{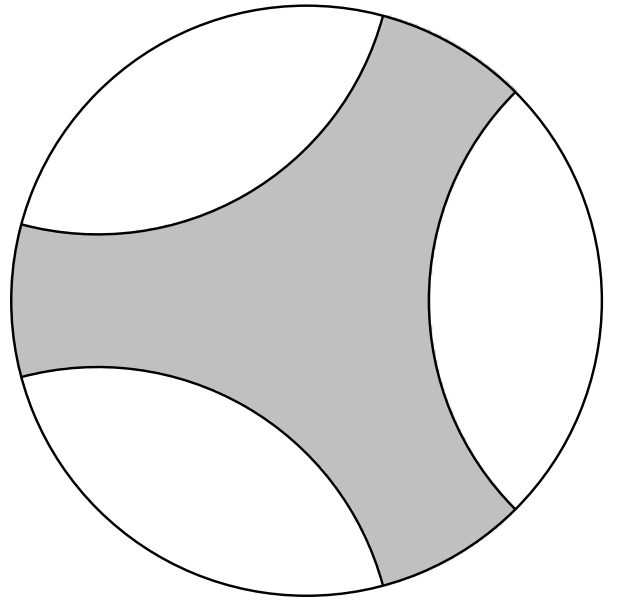}
		\caption{Fundamental domain.}
		\label{schottky:subb}
	\end{subfigure}
	\caption{Example with circles cutting out arcs of angle $\theta = \pi/2$.}
	\label{pi_over_2}
\end{figure}

\subsection{Hecke Groups}\label{sec:IntroHecke}

In Section \ref{ch:Hecke}, we show that our method can still be used in the case where the group in question exhibits both a flare \textit{and a cusp}.
As an example of such groups, we look at infinite covolume \textit{Hecke groups}.
These are groups in the upper half plane model generated by the maps $z \mapsto z + 1$ and $z \mapsto -r^2/z$ for some $r > 0$.

Hecke groups always exhibit a cusp at infinity (because of the shift map).
When $0 < r < 1/2$, they also contain a flare; in other words, this is the condition for which the group has infinite covolume.
One may take the fundamental domain
$$
\mathcal{F} := \left\{ z = x + iy : -\frac{1}{2} < x < \frac{1}{2}, |z| > r \right\}
$$
See Figure \ref{Hecke_orig} for an example.
\begin{figure}
	\centering
	\includegraphics[width=0.7\linewidth]{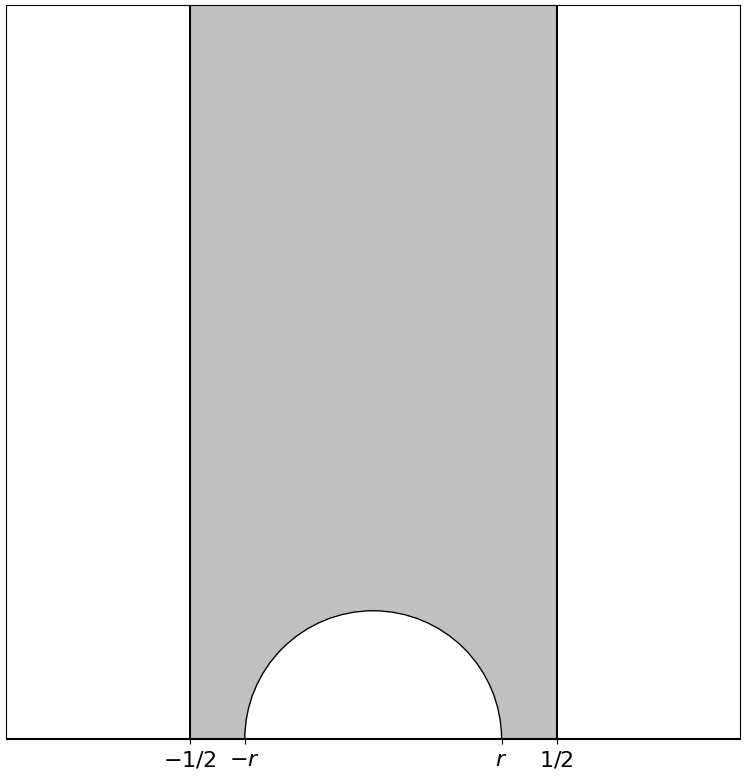}
	\caption{$\mathcal{F}$ for $r = 7/20$}
	\label{Hecke_orig}
\end{figure}

\subsection{Results}

In Section \ref{ch:Results}, we present results in which we compute Maass forms for both symmetric Schottky groups and infinite volume Hecke groups.
We were motivated to study these groups in particular by McMullen, who described an algorithm to compute the Hausdorff dimension of the limit sets of these groups \cite{mcmullen1998hausdorff}.
We also compare our method to McMullen's, as well as discuss how one might verify the results obtained from this work.

In theory, our algorithm applies to an abundance of examples.
In Appendix \ref{ch:Appx_Future}, we present two specific scenarios which would be interesting starting points for further research.

\subsection{Acknowledgments}

This paper is derived from the author's thesis, which was written under the direction of Alex Kontorovich.
The research would not have been possible without Kontorovich and Andreas Str\"ombergsson, whose freely shared notes and code jump-started this work.


\section{Preliminaries}\label{ch:Prelims}

In this section, we begin by describing the setting in which our computations take place.
Then we describe Hejhal's algorithm in its original form.

\subsection{Fuchsian Groups on Hyperbolic 2-Space}\label{sec:FuchsianGroups}

By \textit{hyperbolic 2-space} (or $\mathbb{H}$), we mean the unique simply-connected Riemannian 2-manifold with constant sectional curvature -1.
There are various models of this space, but in this paper we work with the upper half plane model.
Let the set $U^2$ consist of the subset of $\mathbb{C}$ with positive imaginary part
$$
U^2 = \{ x + iy \in \mathbb{C} : y > 0 \}
$$
We endow $U^2$ with the metric
$$
ds^2 = \frac{dx^2 + dy^2}{y^2}
$$
It is well known that $(U^2, ds^2)$ is then isometric to $\mathbb{H}$ (see, e.g. \cite{einsiedler2011ergodic}), so we refer to the metric space as $\mathbb{H}$ from now on.

Next, we note that $\text{PSL}(2, \mathbb{R}) = \text{SL}(2, \mathbb{R})/\{\pm I\}$ acts on $\mathbb{H}$ via M\"obius transformations.
Specifically,
$$
\begin{pmatrix}
	a & b \\
	c & d
\end{pmatrix}(z) =
\frac{az + b}{cz + d}
$$
It is well known that these give all orientation-preserving isometries of hyperbolic 2-space (see, e.g. \cite{iwaniec2021analytic}).
We then define a \textit{Fuchsian group} to be a discrete subgroup of $\text{PSL}(2, \mathbb{R})$.

\subsection{Maass Forms}

Let $\Gamma$ be a Fuchsian group.
Then $f : \mathbb{H} \rightarrow \mathbb{C}$ is a \textit{Maass form} for $\Gamma$ if it is smooth (as a function on $\mathbb{R}^2$), square-integrable on $\Gamma\backslash\mathbb{H}$, invariant under the action of $\Gamma$, and an eigenfunction of the Laplacian.
In other words, $f \in L^2(\Gamma\backslash\mathbb{H}$) must satisfy the following two equations
\begin{equation}\label{invariance}
	f(\gamma z) = f(z) ~\forall~ \gamma \in \Gamma, z \in \mathbb{H}
\end{equation}
\begin{equation}\label{efunc}
	\Delta f(z) = \lambda f(z) ~\text{for some}~ \lambda \in \mathbb{C}
\end{equation}
In the upper half plane model of hyperbolic 2-space, the Laplace operator is given by
\begin{equation}\label{deltaUHP}
	\Delta = -y^2\left( \frac{\partial^2}{\partial x^2} + \frac{\partial^2}{\partial y^2} \right)
\end{equation}

Now suppose $\Gamma$ exhibits a cusp (which is the scenario originally considered by Hejhal).
We can always conjugate the group to position the cusp at infinity with width 1, i.e. we may assume that $\Gamma$ contains the matrix
$\begin{psmallmatrix}
	1 & 1 \\
	0 & 1
\end{psmallmatrix}$.
In this case, a Maass form exhibits a very specific Fourier expansion in its real variable; in particular, the coefficients are \textit{independent} from the imaginary part.

\begin{lem}\label{cusp_lemma}
	Let $\Gamma$ be a Fuchsian group containing the matrix
	$\begin{psmallmatrix}
		1 & 1 \\
		0 & 1
	\end{psmallmatrix}$, and let $f$ be a Maass form for $\Gamma$ with exceptional Laplacian eigenvalue $\lambda = \frac{1}{4} - \nu^2, \nu > 0$.
	Then $f$ has a Fourier expansion of the following form
	\begin{equation}\label{cusp_fourier}
		f(x + iy) = a_0y^{\frac{1}{2} - \nu} + \sum_{n\neq0}a_n\sqrt{y}K_\nu(2\pi|n|y)e(nx)
	\end{equation}
	where $K_\nu$ denotes the $K$-Bessel function.
\end{lem}

\begin{proof}
	By Equation \ref{invariance} $f(z + 1) = f(z)$, and so the Maass form has a Fourier expansion in its real variable
	$$
	f(x + iy) = \sum_{n=-\infty}^{\infty}f_n(y)e(nx)
	$$
	Next we apply Equation \ref{efunc} to the Fourier expansion.
	By uniqueness of Fourier coefficients, we find that
	$$
	\Delta(f_n(y)e(nx)) = \lambda f_n(y)e(nx)
	$$
	for every $n$.
	A straightforward computation (see, e.g. \cite{elstrodt2013groups}) shows that
	$$
	f_0(y) = y^{\frac{1}{2} - \nu} ~~~ f_n(y) = \sqrt{y}K_\nu(2\pi|n|y) ~ (n \neq 0)
	$$
	(The notation $e(z)$ is commonly used to denote $e^{2\pi iz}$).
	In particular, one uses the fact that $f \in L^2(\Gamma\backslash\mathbb{H})$ to rule out other solutions to the second-order differential equations.
	These solutions include the $I$-Bessel function (for $n \neq 0$) and $y^{1/2 + \nu}$ (for $n = 0$), both of which grow too quickly at the cusp for $f$ to be square-integrable.
\end{proof}

So far, we have only assumed that $\lambda$ is complex-valued.
In fact, we can say much more about the possible values.
By applying integration by parts, one can show that the hyperbolic Laplacian is a positive semidefinite operator with respect to the Petersson inner product on $L^2(\Gamma\backslash\mathbb{H})$.
Thus, if $f \in L^2(\Gamma\backslash\mathbb{H})$ is an eigenfunction of the Laplacian with eigenvalue $\lambda = 1/4 - \nu^2$, we must have that $\lambda$ is real and nonnegative.
This can only occur if $\nu = it$ for $t \in \mathbb{R}$ or if $\nu \in \mathbb{R}$ with $|\nu| < 1/2$.
If $\nu$ is real and nonzero, then $\lambda < 1/4$ is called an \textit{exceptional} eigenvalue.

\subsection{Hejhal's Algorithm}\label{sec:HejhalOriginal}

We will now introduce Hejhal's algorithm in the case where the Fuchsian group $\Gamma$ has finite volume and exhibits a cusp at infinity.
As in the previous section, we assume that the cusp has width 1, i.e. $\Gamma$ contains the matrix
$\begin{psmallmatrix}
	1 & 1 \\
	0 & 1
\end{psmallmatrix}$.
A lovely description of the algorithm - along with some technical implementation details - can be found in \cite{booker2006effective}.
This work presents the algorithm along with a computational trick which involves taking a linear combination of the Maass form evaluated at test points along a closed horocycle (this trick was introduced in \cite{hejhal1999eigenfunctions}).
For simplicity, we will present the topic in the manner of Hejhal's original work \cite{hejhal1992eigenvalues}.

First, we consider some high-level pseudocode describing the algorithm.
Then, we give a more detailed explanation of the various parts of the code.
When viewing the pseudocode, recall that we are looking to approximate the eigenvalue $\lambda$ and the Fourier coefficients of a Maass form $f$.
For now, we assume $\lambda$ is not exceptional; that is, $\lambda = 1/4 + t^2$ for some $t \geq 0$.
\begin{verbatim}
	def hejhal(zs1, zs2, zs1_pb, zs2_pb, ts) :
	repeat :
	for t in ts :
	as1[t] = approx_coefficients(zs1, zs1_pb, t)
	as2[t] = approx_coefficients(zs2, zs2_pb, t)
	ts = get_new_ts(as1, as2)
	if STOPPING_CONDITION :
	return mean(ts)
\end{verbatim}
In the code above, \texttt{zs1} and \texttt{zs2} are lists of test points not in some fixed fundamental domain, \texttt{zs1\_pb} and \texttt{zs2\_pb} are lists of the pullbacks of \texttt{zs1} and \texttt{zs2} to that fundamental domain, and \texttt{ts} is a list of initial guesses for $t$.

\textbf{Approximating Coefficients.}
Given a guess for the eigenvalue, we set up and solve a linear system for the first $2M + 1$ Fourier coefficients \textit{assuming our guess is correct}.
Each test point $z$ and pullback $z_*$ give one equation as follows.
Since $f$ is invariant under the group action, $f(z) = f(z_*)$.
We approximate $f$ by cutting off the Fourier expansion \ref{cusp_fourier} to be a sum over $|n| \leq M$.
\begin{equation}
	a_0y^{\frac{1}{2} - \nu} + \sum_{\substack{n\neq0 \\ |n| \leq M}}a_n\sqrt{y}K_{it}(2\pi|n|y)e(nx) \approx \\
	a_0y_*^{\frac{1}{2} - \nu} + \sum_{\substack{n\neq0 \\ |n| \leq M}}a_n\sqrt{y_*}K_{it}(2\pi|n|y_*)e(nx_*)
\end{equation}
Since we have made a guess for $t$, the only unknowns here are the $a_n$'s.
So long as we choose more than $2M + 1$ test points, we will end up with an overdetermined linear system.
We apply least squares to the system to get our approximate Fourier coefficients.

\textbf{Updating Eigenvalue Guesses.}
There are various ways to handle the function \texttt{get\_new\_rs}.
In his original work, Hejhal proposed a straightforward \textit{grid search}.
In this method, one begins with a coarse set of $t$-values, then chooses to search near the value which gives the smallest $L^2$ distance between the approximated coefficients.
To clarify, let $T$ be the set of guesses for $t$ at a fixed step in the algorithm, and write $a_1(t)$ and $a_2(t)$ for the vector of Fourier coefficients approximated using the two sets of test points at a specific $t$ value.
Then, define $t_{best}$ by
$$
t_{best} = \operatorname*{argmin}_{t \in T} ||a_1(t) - a_2(t)||_2
$$
We obtain a new set of $t$ values by taking equally spaced points centered around $t_{best}$; at each step, we take a finer and finer grid, effectively ``zooming in'' on the correct $t$ value.

Another technique (which we adopt in our code) is to utilize the \textit{secant method}.
This starts with just two guesses: $t$ and $t + \delta$ (assume $\delta > 0$).
Consider the guess $t$ first: we approximate the two sets of Fourier coefficients $a_1(t)$ and $a_2(t)$ as before.
We then compute the difference between the coefficients at $\ell$ specific indices (we arbitrarily take $\ell = 4$ in our code):
$$
c_j(t) = a_1(t)[i_j] - a_2(t)[i_j], ~~~ j = 1, \dots, \ell
$$
We then do the same for the guess $t + \delta$.
Note that $c_j(t)$ will be very close to $0$ if $\lambda_t = 1/4 + t^2$ is very close to a true eigenvalue.
Thus, we employ the secant method to search for a zero of the function $c_j$:
$$
t_j = \frac{c_j(t)(t + \delta) - c_j(t + \delta)t}{c_j(t) - c_j(t + \delta)}
$$
Note that the above equation forms $\ell$ new guesses. (The heuristic reason for this is that one of $c_j$'s may be a bad indicator of how good a guess $t$ is for some random reason; it is less likely that all $\ell$ $c_j$'s are bad).
We proceed by setting
\begin{itemize}
	\item $t_{\text{max}} = \max\{t_j\}, t_{\text{min}} = \min\{t_j\}, t_{\text{mid}} = (t_{\text{max}} + t_{\text{min}})/2$
	\item $\delta_{\text{new}} = 5(t_{\text{max}} - t_{\text{min}})$
\end{itemize}
Then we repeat the above process with the new guesses $t_{\text{mid}}$ and $t_{\text{mid}} + \delta_{\text{new}}$.

\textbf{Stopping Conditions.}
The stopping condition depends on the method used for updating the eigenvalue guesses.
In general, the condition should quantify our expectation that we have stopped improving with each guess.
For grid search, we stop when the minimum $L^2$ distance between approximated coefficients at a given step is no smaller than the minimum distance from the previous step.
When using the secant method, we simply check if $\delta_{\text{new}} > \delta/2$.
If so, we don't expect to improve by taking any more steps, so we make our final guess $t_{\text{mid}}$.

\subsection{Choosing Test Points}\label{sec:TestPointsCusp}

In theory, any set of points $\{z_i\} \subset \mathbb{H}$ can be taken as the test points for Hejhal's algorithm, so long as none lie within our chosen fundamental domain.
In practice, however, one finds that certain choices lead to faster convergence.
One such choice is to take our test points along a low-lying horocycle.

Consider the example $\Gamma = \text{SL}(2, \mathbb{Z})$ with the standard fundamental domain
$$
\mathcal{F} = \left\{ z \in \mathbb{H} : |\text{Re}(z)| \leq \frac{1}{2}, |z| \geq 1 \right\}
$$
Recall that horocycles in the upper half-plane model of hyperbolic 2-space are given by horizontal lines.
Thus, a closed horocycle for $\Gamma\backslash\mathbb{H}$ is given by the set
$$
\{ x + iy_0 : -1 \leq x \leq 1 \}
$$
for fixed $y_0 \geq 1$.
However, as we let $y_0$ decrease past 1, this set of points intersects more and more fundamental domains.
See Figure \ref{ed_sl2z}.
\begin{figure}
	\centering
	\includegraphics[width=0.7\linewidth]{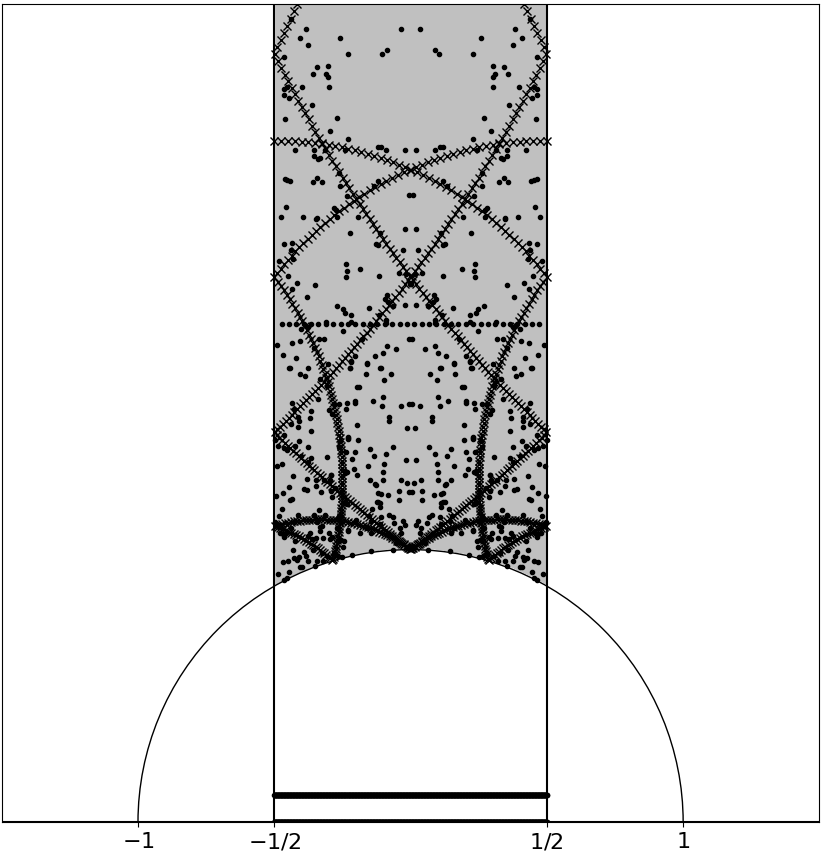}
	\caption{Test points at low-lying horocycles of heights $10^{-1}$ and $10^{-3}$ and their pullbacks given by $\times$'s and dots, resp. Note that the pullbacks of the lower horocycle distribute more evenly about the fundamental domain.}
	\label{ed_sl2z}
\end{figure}

The example above shows one advantage of our choice of test points; namely, the equidistribution of low-lying horocycles.\footnote{the equidistribution result is shown for $\text{SL}(2, \mathbb{Z})$ in \cite{zagier1981eisenstein} and more generally in \cite{sarnak1981asymptotic}}
By choosing a set of points which intersects a large number of fundamental domains, we ensure that a variety of group elements of $\Gamma$ will be involved in pulling our test points back to the fundamental domain.
In other words, the linear system we are setting up for Hejhal's algorithm will include rich \textit{group information} about $\Gamma$.

Another consideration when choosing test points is the convergence of Fourier series.
Recall from Lemma \ref{cusp_lemma} that we are considering Maass forms with a cuspidal expansion
\begin{equation}\label{just_the_series}
	a_0y^{\frac{1}{2} - \nu} + \sum_{n\neq0}a_n\sqrt{y}K_\nu(2\pi|n|y)e(nx)
\end{equation}
It is known that the $K$-Bessel function decays exponentially as its argument goes to infinity (see, e.g., \cite{NIST:DLMF} section 10.30), and thus the entire Fourier coefficient $a_n\sqrt{y}K_\nu(2\pi|n|y)$ decays exponentially with $n$.\footnote{we show in section \ref{sec:Asymptotics} that the $a_n$'s do not grow fast enough to invalidate this statement}
However, this decay is faster for \textit{larger} $y$.
In particular, as $y \rightarrow 0$, we need more and more Fourier coefficients in \ref{just_the_series} to approximate the infinite sum.
This leads us to introduce the notion of \textit{admissibility}.

Consider the same set up as in Lemma \ref{cusp_lemma}, and fix a bound $y_0 > 0$.
We say that a point $z \in \mathbb{H}$ is \textit{admissible with respect to the expansion} \ref{cusp_fourier} if $\text{Im}(z) \geq y_0$.
Since the $K$-Bessel function decays exponentially as its argument grows, $K_\nu(2\pi|n|\text{Im}(z))$ decays at least as fast as $K_\nu(2\pi|n|y_0)$ for any admissible point $z$.

We make one final note about test points in the presence of a cusp.
It turns out that computing the $K$-Bessel function to high accuracy for \textit{large} argument is computationally expensive.
So one sees that there is a trade-off between efficient approximation of the Fourier series (slow when $y$ is small) and efficient computation of $K_\nu(2\pi|n|y)$ (slow when $y$ is large).
For a more in-depth discussion on this, see Section 2 of \cite{booker2006effective}.
We adopt some of their methods in our code.

\section{Hejhal's Algorithm in a Flare Domain}\label{ch:HejhalInFlare}

A main contribution of this work is to remove the requirement of a cuspidal expansion in order to apply Hejhal's algorithm.
In place of this, we require only that the group is finitely generated and non-elementary.
This leads to what we call a \textit{flare expansion}, which is a Fourier expansion of the Maass form in polar coordinates.
In this section, we work out the details of this expansion and describe the modifications needed to apply it in Hejhal's algorithm.

\subsection{Conjugating to a Flare Domain}\label{sec:FlareDef}

In Section \ref{sec:IntroFlares}, we claimed that infinite volume Fuchsian groups which are finitely generated and Zariski dense can be conjugated by a matrix in $\text{SL}(2, \mathbb{R})$ so that the new group admits a fundamental domain containing a flare.
In this section, we show a simple method for finding such a matrix.
This method presupposes the existence of a hyperbolic matrix in the group whose axis\footnote{Recall that the \textit{axis} of a hyperbolic matrix is the geodesic whose endpoints are the fixed points of the matrix.} cuts off the flare at right angles.
This condition is satisfied for both symmetric Schottky groups and infinite volume Hecke groups, so the method presented below is widely applicable.

\begin{lem}
	Let $\Gamma$ be an infinite volume Fuchsian group.
	Let $\mathcal{F}$ be a fundamental domain for $\Gamma$, and suppose it contains an infinite volume portion bordering the real axis.
	Suppose further that $\Gamma$ contains a matrix $A$ whose axis intersects the two geodesics at either side of this infinite volume portion at right angles (see Figure \ref{pre_flare}).
	Then there exists a matrix $U \in \text{SL}(2, \mathbb{R})$ such that $U\Gamma U^{-1}$ admits a flare domain.
	In particular, $U$ may be chosen so that $U\mathcal{F}$ is a flare domain for $U\Gamma U^{-1}$.
\end{lem}

\begin{proof}
	As shown in Figure \ref{pre_flare}, we label the endpoints of the axis of $A$ $z_1$ and $z_2$, $z_1 < z_2$, and we call the rightmost point of the first geodesic $t$.
	\begin{figure}[h]
		\centering
		\includegraphics[width=0.9\linewidth]{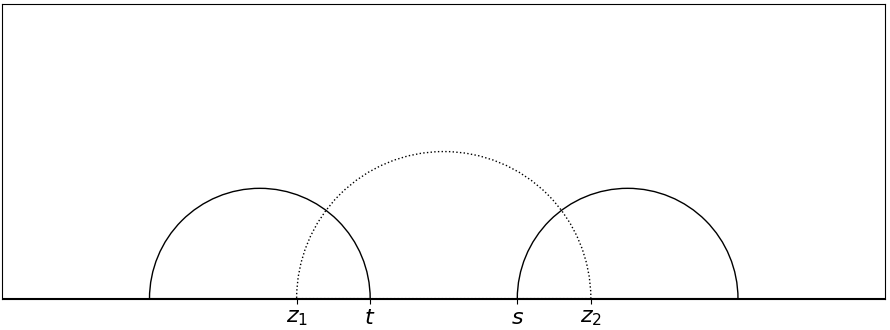}
		\caption{The infinite portion of $\mathcal{F}$. Note that we are assuming no other walls of $\mathcal{F}$ have an endpoint on the real axis between $t$ and $s$.}
		\label{pre_flare}
	\end{figure}
	Now define the M\"obius transformation
	$$
	U(z) = \left( \frac{t - z_2}{t - z_1} \right) \frac{z - z_1}{z - z_2}
	$$
	This function is chosen so that
	$$
	U(z_1) = 0 ~~~~~ U(z_2) = \infty ~~~~~ U(t) = 1
	$$
	Therefore, the axis of $A$ is mapped to the imaginary axis.
	Moreover, since the two bounding geodesics met this axis at right angles, their images under $U$ must meet the imaginary axis at right angles.
	This is only possible if they are centered at the origin, and so we obtain the flare domain shown in Figure \ref{post_flare}.
	\begin{figure}[h]
		\centering
		\includegraphics[width=0.9\linewidth]{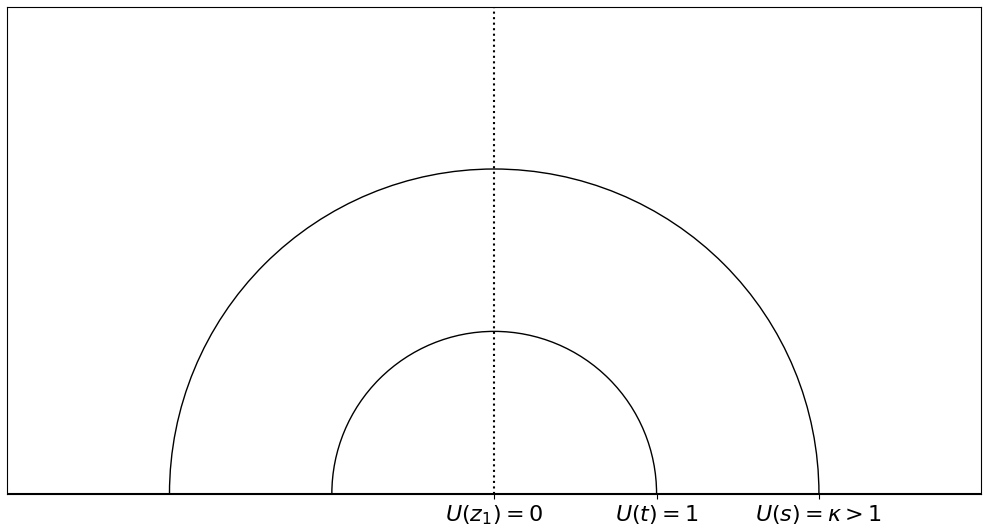}
		\caption{A flare domain.}
		\label{post_flare}
	\end{figure}
	Note that $z_1 < t < s < z_2$ implies that $U(s) > 1$, and so the figure is accurate.
\end{proof}

\subsection{Fourier Expansion in a Flare}

Just as in the cuspidal case, a Maass form exhibits a very specific Fourier expansion in a flare.
In this case, the expansion is in terms of the radial variable $r$, and the coefficients are \textit{independent} from the angle $\theta$.
(This result is shown elsewhere - for example, see Chapter 5 in \cite{MR2710911} - but we include our own proof here for clarity).

\begin{lem}\label{flare_lemma}
	Let $\Gamma$ be an infinite volume Fuchsian group containing the matrix
	$\begin{psmallmatrix}
		\sqrt{\kappa} & 0 \\
		0 & \sqrt{\kappa}^{-1}
	\end{psmallmatrix}$ for some $\kappa > 1$ and which admits a flare domain of width $\kappa$.
	Let $\phi$ be a Maass form for $\Gamma$ with Laplacian eigenvalue $\lambda = \frac{1}{4} - \nu^2$.
	Then $\phi$ has a Fourier expansion in polar coordinates of the following form
	\begin{equation}\label{flare_fourier}
		\phi(r, \theta) = \sum_{n\in\mathbb{Z}}b_n\sqrt{\sin\theta}P^{-\nu}_{\mu_n}(\cos\theta)e\left( n\frac{\log r}{\log \kappa} \right)
	\end{equation}
	where
	$$
	\mu_n = -\frac{1}{2} + \frac{2\pi in}{\log \kappa}
	$$
	and $P_\mu^{-\nu}$ is the associated Legendre function of the first kind.
\end{lem}

\begin{proof}
	Since $\Gamma$ contains
	$\begin{psmallmatrix}
		\sqrt{\kappa} & 0 \\
		0 & \sqrt{\kappa}^{-1}
	\end{psmallmatrix}$
	and $\phi$ is invariant under the group action, we have that $\phi(\kappa z) = \phi(z)$ for all $z \in \mathbb{H}$.
	If we write $\phi = \phi(r, \theta)$, where $(r, \theta)$ are the usual polar coordinates, then $\phi$ is invariant under the map $r \mapsto \kappa r$.
	Thus, $\phi$ has a logarithmic Fourier expansion
	$$
	\phi(r, \theta) = \sum_{n\in\mathbb{Z}}g_n(\theta)e\left( n\frac{\log r}{\log \kappa} \right)
	$$
	Next, we would like to apply the equation $\Delta\phi = \lambda\phi$ to this expansion to find an explicit description of $g_n(\theta)$.
	By a change of variables, one finds $\Delta$ in our new $(r, \theta)$ coordinates.
	\begin{equation}\label{polarL-B}
		\Delta = -\sin^2\theta\left(r^2\frac{\partial^2}{\partial r^2} + r\frac{\partial}{\partial r} + \frac{\partial^2}{\partial\theta^2}\right)
	\end{equation}
	Thus, by uniqueness of Fourier coefficients, $g_n(\theta)$ satisfies the differential equation
	$$
	\sin^2\theta\left( g_n''(\theta) + \left( \frac{2\pi in}{\log k} \right)^2g_n(\theta) \right) = s(s-1)g_n(\theta)
	$$
	A fundamental set of solutions to this differential equation is given by
	$$
	g_n(\theta) = \sqrt{\sin\theta}\mathcal{L}^{-\nu}_{\mu_n}(\cos\theta)
	$$
	where
	$$
	\mu_n = -\frac{1}{2} + \frac{2\pi in}{\log \kappa},
	$$
	and $\mathcal{L} \in \{P, Q\}$.
	That is, the equation exhibits one solution involving the Legendre $P$ function and another linearly independent solution involving the Legendre $Q$ function.\footnote{a helpful reference for special functions is the Digital Library of Mathematical Functions \cite{NIST:DLMF}}
	This allows us to conclude that
	\begin{equation}\label{fourierPQ}
		\phi(r, \theta) = \sum_{n\in\mathbb{Z}}\left(a_n\sqrt{\sin\theta}Q^{-\nu}_{\mu_n}(\cos\theta) + b_n\sqrt{\sin\theta}P^{-\nu}_{\mu_n}(\cos\theta)\right)e\left( n\frac{\log r}{\log \kappa} \right)
	\end{equation}
	
	We are able to remove the $\mathcal{L} = Q$ solution using the fact that $\phi$ must be in $L^2(\Gamma\backslash\mathbb{H})$.
	Since $\mathcal{F}$ contains a flare, we have that
	\begin{equation}\label{inf_bound}
		\infty > \int_\mathcal{F}|\phi^2|\frac{drd\theta}{r\sin^2\theta} > \int_{0}^{\alpha}\int_{1}^{\kappa}|\phi^2|\frac{drd\theta}{r\sin^2\theta}
	\end{equation}
	for some $\alpha > 0$, where $\alpha$ is chosen small enough that the flare
	$$
	\{ z \in \mathbb{H} : 1 < |z| < \kappa, \arg z < \alpha \}
	$$
	is entirely contained in $\mathcal{F}$.
	Applying Parseval to the expansion \ref{fourierPQ}, one finds that
	$$
	\int_{0}^{\alpha}\int_{1}^{\kappa}|\phi^2|\frac{drd\theta}{r\sin^2\theta} =
	\log\kappa\sum_{n\in\mathbb{Z}}\int_{0}^{\alpha}\left| a_nQ^{-\nu}_{\mu_n}(\cos\theta) + b_nP^{-\nu}_{\mu_n}(\cos\theta) \right|^2\frac{d\theta}{\sin\theta}
	$$
	For our values of $\nu$ and $\mu_n$, $P_{\mu_n}^{-\nu}(x)$ is bounded as $x \rightarrow 1$, but $Q_{\mu_n}^{-\nu}(x)$ is unbounded.
	So the size of $\left| a_nQ^{-\nu}_{\mu_n}(\cos\theta) + b_nP^{-\nu}_{\mu_n}(\cos\theta) \right|$ is dominated by $Q_{\mu_n}^{-\nu}$.
	Indeed, for any fixed $n$, $Q_{\mu_n}^{-\nu}(\cos\theta)$ grows quickly enough as $\theta \rightarrow 0$ that the integral above diverges (see \cite{NIST:DLMF} 14.8).
	This would contradict the bound in \ref{inf_bound}, and so we have that $a_n = 0$ for all $n \in \mathbb{Z}$.
\end{proof}

\subsection{The New Algorithm}\label{sec:NewAlg}

To apply Hejhal's algorithm in a flare domain, we simply replace the cuspidal expansion in the original algorithm with the flare expansion.
In other words, the only function which needs to be modified is \texttt{approx\_coefficients}.
(We refer the reader to Section \ref{sec:HejhalOriginal} to review our version of Hejhal's algorithm).

Given a set of test points $\{z_i\}$ and their pullbacks $\{z^*_i\}$ to the fundamental domain, we first compute the corresponding points in polar coordinates $\{(r_i, \theta_i)\}$ and $\{(r_i^*, \theta_i^*)\}$ in the group containing a flare.
To clarify, suppose our original group $\Gamma$ contains a hyperbolic matrix $\gamma_\kappa$ such that
$$
A\gamma_\kappa A^{-1} =
\begin{pmatrix}
	\sqrt{\kappa} & 0 \\
	0 & \frac{1}{\sqrt{\kappa}}
\end{pmatrix}
$$
for some $\kappa > 1$ and $A \in \text{SL}(2, \mathbb{R})$.
Then of course the group $A\Gamma A^{-1}$ contains a diagonal matrix and so exhibits a flare domain.
Noticing that $A\Gamma A^{-1}$ acts on $A\mathbb{H}$ in the same manner as $\Gamma$ acts on $\mathbb{H}$, we choose to take $(r_i, \theta_i)$ and $(r_i^*, \theta_i^*)$ to be the polar coordinates of the point $A(z_i)$ and $A(z_i^*)$, resp.

Now that we have the test points and their pullbacks in polar coordinates, we approximate the Fourier coefficients of the Maass form as before - by truncating the Fourier expansion to a finite sum and setting up a linear system.
Given a guess for the eigenvalue $\lambda = 1/4 - \nu^2$, we get one equation from each test point and its pullback:
\begin{equation}
	\sum_{|n| \leq M}b_n\sqrt{\sin\theta_i}P^{-\nu}_{\mu_n}(\cos\theta_i)e\left( n\frac{\log r_i}{\log \kappa} \right) \approx
	\sum_{|n| \leq M}b_n\sqrt{\sin\theta_i^*}P^{-\nu}_{\mu_n}(\cos\theta_i^*)e\left( n\frac{\log r_i^*}{\log \kappa} \right)
\end{equation}
So long as our number of test points is greater than $2M + 1$, we will have an overdetermined system from which we may approximate the coefficients $b_n$ via least squares.
Since we have a method for approximating the coefficients given a guess for $\nu$, we may proceed with the algorithm as before to update our guess and converge on a true eigenvalue.

\subsection{Choosing Test Points for a Flare Domain}\label{sec:TestPointsFlare}

As in the case discussed in Section \ref{sec:TestPointsCusp}, one finds that certain choices of test points lead to faster convergence of Hejhal's algorithm in a flare.
In this section, we argue that taking points along a ray at a fixed angle constitutes one such choice.

In Figure \ref{ed_schottky}, we display a flare domain for the symmetric Schottky group with parameter $\theta = \pi/2$ (see Section \ref{ch:Schottky} for details on this group).
\begin{figure}
	\centering
	\includegraphics[width=\linewidth]{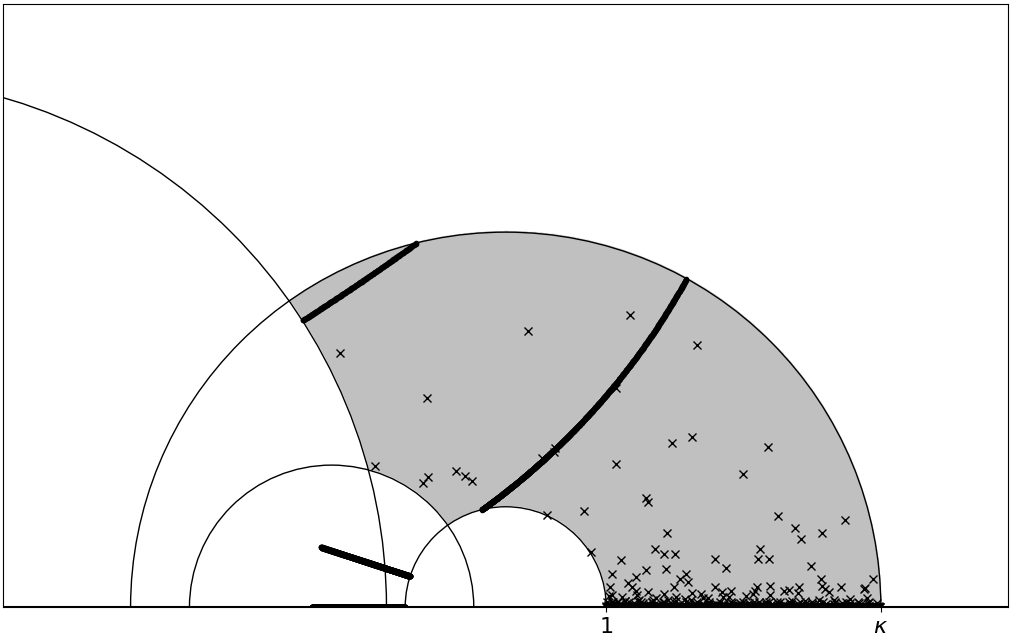}
	\caption{Test points at rays of angles $9999\pi/10000$ and $9\pi/10$ and their pullbacks given by $\times$'s and dots, resp. Note that the pullbacks coming from the larger angle distribute more sparsely about the fundamental domain.}
	\label{ed_schottky}
\end{figure}
The test points are chosen along a ray at a fixed angle $\varphi$, spaced evenly on a logarithmic scale from $1$ to $\sqrt{\kappa}$.
$$
\left\{ \kappa^{j/2000}e^{i\varphi} \right\}_{j = 1}^{1000}
$$
We only choose test points with norm between $1$ and $\kappa$ since the flare expansion is invariant under scaling by $\kappa$ (indeed, the fact that the group contains multiplication by $\kappa$ is what gives us a flare domain in the first place!).
In this particular example, we further restrict the test points to have norm between $1$ and $\sqrt{\kappa}$ because the group has a double cover which identifies points reflected across the geodesic $\{z \in \mathbb{H} : |z| = \sqrt{\kappa}\}$.
This cover is discussed in detail in Section \ref{ch:Schottky}.

As is seen in Figure \ref{ed_schottky}, our example exhibits equidistribution of pullback points as the angle of the ray approaches $\pi$.
This is a general result; see \cite{kelmer2018effective}.
In this way, test points chosen along rays in the flare expansion are similar to the low-lying horocycles in a cuspidal expansion.

As before, we must also consider the convergence of Fourier series.
Recall from Lemma \ref{flare_lemma} that Maass forms in a flare domain have the expansion
\begin{equation}\label{just_flare_series}
	\sum_{n\in\mathbb{Z}}b_n\sqrt{\sin\theta}P^{-\nu}_{\mu_n}(\cos\theta)e\left( n\frac{\log r}{\log \kappa} \right)
\end{equation}
In Section \ref{sec:Asymptotics}, we will see that the entire Fourier coefficient $b_n\sqrt{\sin\theta}P^{-\nu}_{\mu_n}(\cos\theta)$ decays exponentially with $n$.
Moreover, the proof of Lemma \ref{asymptotic_flare} shows that the decay is faster for smaller $\theta$.
Just as in the cuspidal case, this leads us to introduce the notion of \textit{admissibility} with respect to a flare expansion.

Consider the same set up as in Lemma \ref{flare_lemma}, and fix a bound $\alpha_0$ with $0 < \alpha_0 < \pi$.
We say that a point $z \in \mathbb{H}$ is \textit{admissible with respect to the expansion} \ref{flare_fourier} if $\arg(z) \leq \alpha_0$.
Unlike the cuspidal case, we are not forced to choose between fast convergence of the Fourier series and efficient computation of the Legendre-$P$ function.
Indeed, numerical evaluation of $P^{-\nu}_{\mu_n}(\cos\theta)$ is more efficient for smaller $\theta$, and it becomes more time-consuming as $\theta$ increases to $\pi$ (see \cite{NIST:DLMF} Section 14).

\subsection{Multiple Cusps and Flares}\label{sec:MultiExpansion}

In the presence of multiple cusps, it has been found to be necessary to simultaneously solve for the Fourier coefficients at every cusp.
See \cite{selander2002sextic} and \cite{stromberg2012vi} for some examples of this.
In this section, we explain how this can be extended to cases involving multiple cusps and/or multiple flares.

To that note, suppose $\Gamma$ is a Fuchsian group involving $c \geq 0$ cusps and $\ell \geq 0$ flares, with $c + \ell \geq 1$.
We will call the Maass form we are searching for $f$ and its eigenvalue $\lambda = 1/4 - \nu^2$.
Let $A_i \in \text{SL}(2, \mathbb{R})$ reposition the $i^{th}$ cusp to infinity with width $1$; i.e. $A_i\Gamma A_i^{-1}$ has a fundamental domain with a width 1 cusp at infinity, and $A_i$ maps the $i^{th}$ cusp to infinity.
Then by Lemma \ref{cusp_lemma}, $f$ has the following cuspidal expansion for each $i$.
\begin{equation}\label{cuspAi}
	f(z) = a^{(i)}_0\text{Im}(A_iz)^{\frac{1}{2} - \nu} + \sum_{n\neq0}a^{(i)}_n\sqrt{\text{Im}(A_iz)}K_\nu(2\pi|n|\text{Im}(A_iz))e(n\text{Re}(A_iz))
\end{equation}
We follow a similar pattern for the flares.
Let $B_j \in \text{SL}(2, \mathbb{R})$ reposition the $j^{th}$ flare to a flare domain of size $\kappa_j > 1$ (see Section \ref{sec:FlareDef}).
Then by Lemma \ref{flare_lemma}, $f$ has the following flare expansion for each $j$.
\begin{equation}\label{flareBj}
	f(z) = \sum_{n\in\mathbb{Z}}b^{(j)}_n\sqrt{\sin\theta(B_jz)}P^{-\nu}_{\mu_n}(\cos\theta(B_jz))e\left( n\frac{\log r(B_jz)}{\log \kappa} \right)
\end{equation}
where $\mu_n$ is defined in Lemma \ref{flare_lemma} and $(r(z), \theta(z))$ denotes the usual polar coordinates of a point $z \in \mathbb{H}$.

The algorithm then proceeds with just two major changes.
First, we introduce two parameters $M_C, M_F > 0$.
All cuspidal expansions are cut off at $|n| \leq M_C$, and all flare expansions are cut off at $|n| \leq M_F$.
Thus, we have a finite number of coefficients to solve for; namely, $c(2M_C + 1)$ cuspidal coefficients and $\ell(2M_F + 1)$ flare coefficients.

Second, we must choose which expansion to use for any given test point and its pullback.
We make use of the notion of \textit{admissibility}, defined for cusps in Section \ref{sec:TestPointsCusp} and for flares in Section \ref{sec:TestPointsFlare}.
Given $z \in \mathbb{H}$, we apply the following rules to choose an expansion.\footnote{one may suggest using all admissible expansions for a given test point, thus giving multiple equations per point; it was found experimentally that this slows the convergence of the algorithm}
\begin{enumerate}
	\item if $z$ is admissible for exactly one of the $c + \ell$ expansions, we use that expansion
	\item if $z$ is admissible for more than one cuspidal expansion but none of the flare expansions, we use the cuspidal expansion for which $\text{Im}(A_iz)$ is largest
	\item if $z$ is admissible for one or more cuspidal expansions and one or more flare expansions, we always use one of the flare expansions
	\item if $z$ is admissible for more than one flare expansion, we use the one for which $\theta(B_iz)$ is minimal
	\item if $z$ is not admissible for any expansion, we do not use it as a test point
\end{enumerate}
To justify rule 2, think of this as $A_iz$ being closest to infinity, and therefore $z$ being closer to the $i^{th}$ cusp than to any other cusp.
Rule 3 is arbitrary; the author suspects that replacing this rule with its opposite would neither aid nor impair convergence.
Finally, rule 4 is justified in the same manner as rule 2.

As a quick example, suppose our rules lead us to choose cuspidal expansion $i$ for a test point $z$ and flare expansion $j$ for its pullback $z^*$.
Then this point-pullback pair contributes the following equation to the linear system
\begin{multline}
	a^{(i)}_0\text{Im}(A_iz)^{\frac{1}{2} - \nu} + \sum_{\substack{n\neq0 \\ |n| \leq M_C}}a^{(i)}_n\sqrt{\text{Im}(A_iz)}K_\nu(2\pi|n|\text{Im}(A_iz))e(n\text{Re}(A_iz)) \\
	\approx
	\sum_{|n| \leq M_F}b^{(j)}_n\sqrt{\sin\theta(B_jz^*)}P^{-\nu}_{\mu_n}(\cos\theta(B_jz^*))e\left( n\frac{\log r(B_jz^*)}{\log \kappa} \right)
\end{multline}
Since each test point contributes one equation, we need more than $c(2M_C + 1) + \ell(2M_F + 1)$ test points to achieve an overdetermined system.
Note also that each expansion must appear at least once in the system; otherwise, we have no information about those Fourier coefficients.

Before moving on, we note that in the presence of multiple expansions, one may also take test points \textit{within our chosen fundamental domain}.
To get an equation in this scenario, we use the rules above to choose two expansions at which the point is admissible.
Then, we set the expansions at that point equal to each other.
Note that if a test point is taken inside our fundamental domain, it must be admissible with respect to at least two expansions.

\subsection{Bounds on Fourier Coefficients}\label{sec:Asymptotics}

For the cuspidal and flare expansions to converge, the Fourier coefficients must of course decay as $|n| \rightarrow \infty$.
In fact, it is possible to get upper bounds on these rates of decay, as we show in this section.
Having these rates is useful for Hejhal's algorithm, since we can rescale the coefficients in our finite linear system to increase the stability of the least squares solution.

Before stating a bound for the cuspidal case, we state a result which will be needed in the proof.
We consider a finitely generated, discrete subgroup $\Gamma < \text{SL}(2, \mathbb{Z})$.
By Patterson \cite{patterson1975laplacian} and Lax-Philips \cite{lax1982asymptotic}, the Laplacian acting on $L^2(\Gamma\backslash\mathbb{H})$ admits finitely many eigenvalues in $[0, 1/4)$.
We denote these eigenvalues
$$
0 \leq \lambda_0 < \lambda_1 < \cdots < \lambda_k < \frac{1}{4}
$$
and we write $\phi_j$ for the eigenfunction corresponding to $\lambda_j$.
Let us assume the eigenfunctions are scaled so that $||\phi_j||_2 = 1$.

Let $s_j > 1/2$ satisfy $\lambda_j = s_j(1 - s_j)$ for $1 \leq j \leq k$.
Then we have the following result on equidistribution of low-lying horocycles.
\begin{thm}\label{equidistribution_them}
	Fix notation as above, and let $\psi \in C^\infty(\Gamma\backslash\mathbb{H})$. Then
	$$
	\int_{0}^{1}\psi(x + iy)dx =
	C\langle \psi, \phi_0 \rangle y^{1-\delta} + O\left( y^{1 - s_1} \right)
	$$
	as $y \rightarrow 0$, where $\delta$ is the Hausdorff dimension of the limit set of $\Gamma$, $C > 0$ is independent of $\psi$, and the implied constant depends on a Sobolev norm for $\psi$.
\end{thm}

\begin{proof}
	See \cite{kontorovich2012almost} Theorem 1.11.
\end{proof}

Next, we provide a bound for the Fourier coefficients in a cuspidal expansion.

\begin{lem}
	Let $f$ be a Maass form with Laplacian eigenvalue $\lambda = 1/4 - \nu^2$ for a Fuchsian group $\Gamma$.
	Suppose $\Gamma$ has a width 1 cusp at infinity for which $f$ has the cuspidal expansion
	$$
	f(x + iy) = a_0y^{\frac{1}{2} - \nu} + \sum_{n\neq0}a_n\sqrt{y}K_\nu(2\pi|n|y)e(nx)
	$$
	Then the Fourier coefficients satisfy
	\begin{equation}\label{cusp_bound}
		|a_n| < C|n|^{\delta - 1/2}
	\end{equation}
	where $\delta$ is the Hausdorff dimension of the limit set of $\Gamma$ and $C > 0$ is a constant independent of $n$.
\end{lem}

\begin{proof}
	We will use the fact that $f \in L^2(\Gamma\backslash\mathbb{H})$ to get a bound on the $a_n$'s.
	For simplicity, let us assume for the remainder of the argument that $n > 0$ (the $n < 0$ case follows similarly).
	First, we multiply both sides of the cuspidal expansion by $1/y^2$ and integrate over $y$ in $1/n$ to $2/n$ (assuming of course that $n \neq 0$).
	This bounds $K_{\nu}(2\pi ny)$ between two constants, and hence we conclude that
	$$
	\left|a_n\int_{1/n}^{2/n}\sqrt{y}\frac{dy}{y^2}\right| \leq C_0\left|\int_{1/n}^{2/n}\int_{0}^{1}f(x+iy)e(-nx)\frac{dxdy}{y^2}\right|
	$$
	where $C_0 > 0$ is a constant independent of $n$.
	The integral on the left hand side is a positive constant times $\sqrt{n}$, and so we have
	\begin{equation}\label{firstCuspBound}
		|a_n| \leq
		C_1n^{-1/2}\int_{1/n}^{2/n}\int_{0}^{1}|f(x+iy)|\frac{dxdy}{y^2}
	\end{equation}
	where $C_1 > 0$ is a constant independent of $n$.
	
	Next, we apply Theorem \ref{equidistribution_them} to the inner integral on the right hand side of \ref{firstCuspBound}.
	Taking $\psi = |f|$, this implies that
	$$
	\int_{0}^{1}|f(x + iy)|dx =
	C\langle |f|, \phi_0 \rangle y^{1-\delta} + O\left( y^{1 - s_1} \right)
	$$
	(Recall that $dxdy/y^2$ is the Haar measure on $\mathbb{H}$).
	By Cauchy-Schwartz, the inner product is bounded by the $L_2$ norm of $f$.
	Since $s_1 < \delta$, the previous equation implies that
	\begin{equation}\label{xIntBound}
		\int_{0}^{1}|f(x + iy)|dx \ll_f y^{1 - \delta}
	\end{equation}
	Putting together the bounds (\ref{firstCuspBound}) and (\ref{xIntBound}), we get
	$$
	|a_n| \ll_f n^{-1/2}\int_{1/n}^{2/n}y^{-1-\delta}dy
	$$
	from which immediately follows the desired result.
\end{proof}

Next, we provide a bound for the Fourier coefficients in a flare expansion.

\begin{lem}\label{asymptotic_flare}
	Let $f$ be a Maass form with exceptional eigenvalue $\lambda = 1/4 - \nu^2$ for a Fuchsian group $\Gamma$, and assume $f$ is scaled so that $||f|| = 1$.
	Suppose $\Gamma$ has a flare domain with a flare of size $\kappa > 1$ for which $f$ has the expansion
	$$
	f(z) = \sum_{n\in\mathbb{Z}}b_n\sqrt{\sin\theta}P^{-\nu}_{\mu_n}(\cos\theta)e\left( n\frac{\log r}{\log \kappa} \right)
	$$
	where $(r, \theta)$ denote the polar coordinates for $z$ and $\mu_n$ is taken as in Lemma \ref{flare_lemma}.
	Then the Fourier coefficients satisfy
	$$
	b_n \ll n^se^{-\pi n\alpha/\log\kappa}
	$$
	where $\lambda = s(1 - s), 1/2 \leq s \leq 1$ and $0 < \alpha < \pi$ is a constant independent of $n$.
\end{lem}

\begin{proof}
	This statement is proven in \cite{kontorovich2012almost} Proposition A.5.
\end{proof}

Let us make a few quick remarks on Lemma \ref{asymptotic_flare}.
First, recall that the assumption that $\lambda$ is \textit{exceptional} is equivalent to the statement $\lambda < 1/4$, i.e. that $0 \leq \nu \leq 1/2$.
In this case, we prefer the parameterization $\lambda = s(1 - s)$ where $s = \nu + 1/2$.
This makes the statement of the lemma cleaner and, in the case that $\lambda = \lambda_0$ is the bottom eigenvalue, $s$ is equal to the Hausdorff dimension of the limit set.

Second, note that the quadratic $s(1 - s)$ is decreasing on $1/2 \leq s \leq 1$.
This means that the largest value of $s$ occurs at the smallest eigenvalue.
As we noted in the introduction, this value of $s$ is exactly equal to $\delta$, the Hausdorff dimension of the limit set of $\Gamma$.
So we may replace the bound in Lemma \ref{asymptotic_flare} with the weaker bound
\begin{equation}\label{delta_bound}
	b_n \ll n^\delta e^{-\pi n\alpha/\log\kappa}
\end{equation}
if so desired.
For example, if we know the Hausdorff dimension $\delta$ but none of the other eigenvalues, this bound may be the best we can find.

Finally, we claim that this result extends to non-exceptional eigenvalues using the Hausdorff dimension in the bound as in (\ref{delta_bound}).
The proof of this follows in the same lines as the one given in \cite{kontorovich2012almost} Proposition A.5.
Instead of the parameterization $\lambda = s(1 - s)$, one uses $\lambda = (1/2 - ir)(1/2 + ir)$ for some $r \geq 0$.

\subsection{Notes on Base Eigenfunctions}\label{sec:BaseEfuncs}

Before moving on to our examples in Sections \ref{ch:Schottky} and \ref{ch:Hecke}, we discuss the special properties of eigenfunctions associated with the smallest eigenvalue.
As discussed in Section \ref{sec:FuchsianGroups}, all eigenvalues of the Laplacian are nonnegative.
We call the smallest eigenvalue the \textit{base eigenvalue} and the eigenfunctions associated with it \textit{base eigenfunctions}.

In the introduction, we already saw one property of the base eigenvalue, namely its quadratic relationship with the Hausdorff dimension of the limit set of a Fuchsian group.
There are other interesting results on base eigenfunctions that can be utilized to speed up convergence of Hejhal's algorithm in this special case.
We collect such results in the present section, beginning with a basic fact from Patterson-Sullivan theory.

\begin{lem}\label{mult1}
	Let $\Gamma$ be a Fuchsian group with base eigenvalue $\lambda_0$.
	Then $\lambda_0$ has multiplicity 1.
\end{lem}

\begin{proof}
	This is shown in \cite{gamburd2002spectral}.
\end{proof}

For the reader's convenience, we give the standard arguments to prove the rest of the results in this section.

\begin{lem}\label{real-valued}
	Let $\Gamma$ be a Fuchsian group with base eigenvalue $\lambda_0$.
	Then we may take a base eigenfunction to be real-valued.
\end{lem}

\begin{proof}
	If $f$ is an eigenfunction of the Laplace-Beltrami operator, then $\bar{f}$ is also an eigenfunction with the same eigenvalue (see Equation \ref{deltaUHP}).
	By Lemma \ref{mult1}, we also know that the base eigenvalue has multiplicity $1$.
	Together, these facts imply that the base eigenfunction $\varphi$ satisfies
	$$
	\varphi(z) = \omega\bar{\varphi}(z)
	$$
	for some $\omega \in \mathbb{C}$.
	In other words, $\arg(\varphi(z))$ is constant.
	Since an eigenfunction is only defined up to multiplication by a nonzero constant, we can multiply by a rotation to assume without loss of generality that $\varphi$ is real-valued.
\end{proof}

\begin{lem}\label{nonneg}
	Let $\Gamma$ be a Fuchsian group with base eigenvalue $\lambda_0$.
	Then we may take a base eigenfunction to be nonnegative.
\end{lem}

\begin{proof}
	Call the base eigenfunction $\varphi$, and assume by Lemma \ref{real-valued} that it is real-valued.
	Recall that the zero set of $\varphi$ consists of a finite union of real analytic curves (for a discussion on these \textit{nodal lines}, see e.g. \cite{donnelly1987nodal}).
	In particular, the nodal lines separate a fundamental domain for $\Gamma\backslash\mathbb{H}$ into regions on which $\varphi$ takes only positive or only negative values.
	Next, note that
	\begin{equation}\label{inf0}
		\lambda_0 = \inf_{\phi \in C_0^\infty(\Gamma\backslash\mathbb{H})}\frac{||\Delta\phi||}{||\phi||}
	\end{equation}
	where $C_0^\infty(\Gamma\backslash\mathbb{H})$ denotes the dense subspace of smooth functions with compact support in $L^2(\Gamma\backslash\mathbb{H})$.
	This is true for all positive definite operators.
	
	To finish the argument, we consider the negative-valued sections of $\varphi$ separately from the positive-valued sections.
	If we replace any negative sections with their absolute value, we still have an $L^2$ function.
	By Equation \ref{inf0}, this leaves the eigenvalue unchanged, contradicting Lemma \ref{mult1}.
	The contradiction is avoided of course if there were no negative sections to begin with, and thus we see that $\varphi$ is nonnegative.
	(It also could have been the case that $\varphi$ had no positive-valued sections, but then replacing $\varphi$ with $-\varphi$ gives us the same result).
\end{proof}

\begin{lem}\label{efunc_even}
	Let $\Gamma$ be a Fuchsian group with base eigenvalue $\lambda_0$, and suppose that $\Gamma$ is invariant under conjugation by the map $z \mapsto -\bar{z}$.
	Then the base eigenfunction is even.
\end{lem}

\begin{proof}
	Let $\varphi$ denote the base eigenfunction, and write $c$ for the map $c(z) = -\bar{z}$.
	Since $\Gamma$ is invariant under conjugation by $c$, we have that $\varphi\circ c$ is $\Gamma$-invariant.
	Moreover, two derivatives are taken in $x$ in the hyperbolic Laplacian (see Equation \ref{deltaUHP}).
	Therefore, we see that $\varphi\circ c(x + iy) = \varphi(-x + iy)$ is also an eigenfunction with the same eigenvalue.
	By Lemma \ref{mult1}, the base eigenvalue has multiplicity 1, implying that
	$$
	\varphi(x + iy) = \omega\varphi(-x + iy)
	$$
	for some $\omega \in \mathbb{C}$.
	On the other hand, the functions $\varphi(x + iy)$ and $\varphi(-x + iy)$ match on the line $x = 0$; so we must have that $C = 1$.
\end{proof}

\begin{lem}\label{cusp_real}
	Let $\Gamma$ be a Fuchsian group with a width 1 cusp at infinity, and let $\varphi$ denote the base eigenfunction.
	Suppose further that $\Gamma$ is invariant under conjugation by the map $z \mapsto -\bar{z}$.
	Then the Fourier coefficients in the cuspidal expansion for $\varphi$ may be taken to be real-valued.
\end{lem}

\begin{proof}
	By Lemma \ref{cusp_lemma}, we may write the Fourier expansion at the cusp as
	$$
	\varphi(x + iy) = \sum_{n \in \mathbb{Z}}a_nW_n(y)e(nx)
	$$
	where
	$$
	W_n(y) =
	\begin{cases}
		y^{\frac{1}{2} - \nu} & n = 0 \\
		\sqrt{y}K_\nu(2\pi|n|y) & n \neq 0
	\end{cases}
	$$
	and the base eigenvalue is $\lambda_0 = \frac{1}{4} - \nu^2$.
	From Lemma \ref{efunc_even}, we know that $\varphi$ is even.
	Thus,
	$$
	\sum_{n \in \mathbb{Z}}a_nW_n(y)e(nx) = \sum_{n \in \mathbb{Z}}a_nW_n(y)e(-nx)
	$$
	Putting this together with the definition of $W_n(y)$, this immediately implies that $a_n = a_{-n}$.
	This allows us to ``fold'' the Fourier expansion to only allow nonnegative coefficients.
	This proceeds by adding the $n$ and $-n$ terms together and using the fact that
	$$
	e(nx) + e(-nx) = 2\cos(2\pi nx)
	$$
	We can absorb the $2$ into the coefficient $a_n$ and get
	\begin{equation}\label{cusp_folded}
		\varphi(x + iy) = a_0W_0(y) + \sum_{n = 1}^{\infty}a_nW_n(y)\cos(2\pi nx)
	\end{equation}
	Finally, we use the fact that $\varphi$ is real-valued (Lemma \ref{real-valued}) to see that the $a_n$'s must be real as well.
\end{proof}

In Lemma \ref{cusp_real}, we see how we may make Hejhal's algorithm more efficient when searching for the base eigenvalue.
In particular, the cuspidal expansion exhibits the extra structure shown in Equation \ref{cusp_folded}.
We use this folded expansion in the linear system whenever we are approximating the Fourier coefficients of the base eigenfunction.

In Lemma \ref{flare_real} below, we will see that a similar result holds in the flare expansion.
But first, we require a technical lemma on the Legendre $P$-function.

\begin{lem}\label{P_technical}
	Let $\nu, \theta, \kappa \in \mathbb{R}$, $n \in \mathbb{Z}$, and $\mu_n = -1/2 + 2\pi in/\log\kappa$.
	Then
	$$
	P_{\mu_{n}}^{-\nu}(\cos\theta) = P_{\mu_{-n}}^{-\nu}(\cos\theta)
	$$
\end{lem}

\begin{proof}
	Recall that for $-1 < x < 1$, the Legendre $P$ function is defined by
	$$
	P_\mu^\nu(x) = 
	\frac{1}{\Gamma(1 - \nu)}\left( \frac{1 + x}{1 - x} \right)^{\nu/2}\prescript{}{2}{F_1}(\mu + 1, -\mu, 1 - \nu, (1 - x)/2)
	$$
	(see \cite{NIST:DLMF} 14.3.1).
	The hypergeometric function is defined by the series
	$$
	\prescript{}{2}{F_1}(a, b, c, z) = \sum_{n=0}^{\infty}\frac{(a)_n(b)_n}{(c)_n}\frac{z^n}{n!}
	$$
	where $(a)_n$ denotes the Pochammer symbol.
	Now since $\mu_n = -1/2 + 2\pi in/\log\kappa$, we have $\mu_{-n} = \bar{\mu}_n$.
	Tracking this bar through the definitions of the special functions above, we see that
	$$
	P_{\mu_{-n}}^{-\nu}(\cos\theta) = \bar{P}_{\mu_{n}}^{-\nu}(\cos\theta)
	$$
	On the other hand, note that the only piece of $P_{\mu_{n}}^{-\nu}(\cos\theta)$ which is not clearly real is the hypergeometric function
	$$
	\prescript{}{2}{F_1}(\mu_n + 1, -\mu_n, 1 + \nu, (1 - \cos\theta)/2)
	$$
	But in fact this is real, since $\mu_n + 1 = \frac{1}{2} + \frac{2\pi in}{\log\kappa}$, $-\mu_n = \frac{1}{2} - \frac{2\pi in}{\log\kappa}$, and
	$$
	(a + x)_m(a - x)_m = \prod_{k=0}^{m-1}((a + k)^2 - x^2)
	$$
	So since $P_{\mu_{n}}^{-\nu}(\cos\theta)$ is real and $P_{\mu_{-n}}^{-\nu}(\cos\theta) = \bar{P}_{\mu_{n}}^{-\nu}(\cos\theta)$, we have concluded the proof that $P_{\mu_{-n}}^{-\nu}(\cos\theta) = P_{\mu_{n}}^{-\nu}(\cos\theta)$.
\end{proof}

\begin{lem}\label{flare_real}
	Let $\Gamma$ be a Fuchsian group which admits a flare domain of width $\kappa$, and let $\phi(r, \theta)$ denote the base eigenfunction in polar coordinates.
	Suppose further that $\Gamma$ is invariant under conjugation by the map $r \mapsto \kappa/r$.
	Then the Fourier coefficients in the flare expansion for $\phi$ may be taken to be real-valued.
\end{lem}

\begin{proof}
	Write $s$ for the map $s(r, \theta) = (\kappa/r, \theta)$, and suppose the function $f(r, \theta)$ is $\Gamma$-invariant.
	Then since $\Gamma$ is invariant under conjugation by $s$, $f\circ s(r, \theta)$ is $\Gamma$-invariant as well.
	Moreover, it is a simple exercise to check that if $f(r, \theta)$ is an eigenfunction of the hyperbolic Laplacian, then $f\circ s(r, \theta) = f(\kappa/r, \theta)$ is also an eigenfunction with the same eigenvalue (see Equation \ref{polarL-B} for the Laplace-Beltrami operator in polar coordinates).
	Next, by Lemma \ref{flare_lemma}, we may write
	$$
	\phi(r, \theta) = \sum_{n \in \mathbb{Z}}b_n\sqrt{\sin\theta}P^{-\nu}_{\mu_n}(\cos\theta)e\left( n\frac{\log r}{\log\kappa} \right) ~~~\text{where}~~~ \mu_n = -\frac{1}{2} + \frac{2\pi in}{\log \kappa}
	$$
	for the Fourier expansion at a flare.
	Since the base eigenfunction has multiplicity 1 (Lemma \ref{mult1}), and since $\phi(\kappa/r, \theta)$ is also an eigenfunction with the same eigenvalue, we have that
	$$
	\phi(r, \theta) = \omega\phi(\kappa/r, \theta)
	$$
	for some constant $\omega \in \mathbb{C}$.
	But these functions agree at $r = \sqrt{\kappa}$; thus, $\omega = 1$ and $\phi$ is invariant under $r \mapsto \kappa/r$.
	
	Setting the Fourier expansions of $\phi(r, \theta)$ and $\phi(\kappa/r, \theta)$ equal to each other, and using the fact that $e(n\log(\kappa/r)/\log\kappa) = e(-n\log r/\log\kappa)$, we see that
	$$
	b_n\sqrt{\sin\theta}P_{\mu_n}^{-\nu}(\cos\theta) = b_{-n}\sqrt{\sin\theta}P_{\mu_{-n}}^{-\nu}(\cos\theta)
	$$
	By Lemma \ref{P_technical}, $P_{\mu_{n}}^{-\nu}(\cos\theta) = P_{\mu_{-n}}^{-\nu}(\cos\theta)$.
	So in fact, this shows that $b_n = b_{-n}$.
	
	We now wish to ``fold'' the Fourier expansion by adding the $\pm n$ terms together.
	This gives
	$$
	\phi(r, \theta) = b_0\sqrt{\sin\theta}P_{-1/2}^{-\nu}(\cos\theta) +
	\sum_{n=1}^{\infty}b_n\sqrt{\sin\theta}P_{\mu_{n}}^{-\nu}(\cos\theta)\cos\left( 2\pi n\frac{\log r}{\log\kappa} \right)
	$$
	Since $\phi$ is real-valued, and since every other term in the expansion is real, this shows that the $b_n$'s are real as well.
\end{proof}


\section{Hejhal's Algorithm applied to Schottky Groups}\label{ch:Schottky}

In this section, we work out an explicit example of applying our algorithm to a family of groups which exhibit \textit{no cusps}, namely, symmetric Schottky groups.
As the reader will no doubt observe, there are quite a few small details one must get right in order to achieve convergence of the algorithm.
It is the author's hope that this example - along with that given in Section \ref{ch:Hecke} - will aid those who wish to apply the extended algorithm to other Fuchsian groups.
We urge the reader to pay particular attention to use of group covers in this section; if a cover of a Fuchsian group exists, it is likely that the new algorithm will fail to converge unless one moves to the cover.

\subsection{Fundamental Domain for Symmetric Schottky Groups}

Recall from section \ref{sec:IntroSchottky} that we parameterize a symmetric Schottky group by the angle $\theta$ along the unit circle cut out by each of the three circles.
In Figure \ref{pi_over_2}(b), we shade in a fundamental domain for a symmetric Schottky group.
We verify that image in the present section.
Note that the proof below implies a pullback algorithm for finding points in the fundamental domain equivalent to any point in the disk, which is an essential piece of Hejhal's algorithm.

\begin{lem}
	Let $\mathcal{S}$ be a symmetric Schottky group with parameter $\theta < 2\pi/3$.
	Denote the 3 circles which generate the reflection group $R_1, R_2, R_3$.
	Then a fundamental domain for $\mathcal{S}$ is the set of points in the interior of the unit circle which fall in the exterior of all three circles.
\end{lem}

\begin{proof}
	We will use $R_i$ to denote both the circle and the reflection through it; the meaning should be clear from context.
	Let
	$$
	\gamma = R_{i_1}\cdots R_{i_k}
	$$
	be a reduced word of length $k$ in the reflections (by \textit{reduced} we mean that $R_{i_j} \neq R_{i_{j+1}}$ for all $j$, since reflections are involutions).
	Note that every element of $\mathcal{S}$ can be written as such a reduced word.
	Now suppose $z$ is any element of the fundamental domain.
	If $k = 1$, then it is clear that $\gamma(z)$ is inside the circle $R_{i_1}$.
	By induction, it is then easy to see that for $k \geq 1$, $\gamma(z)$ is always in the circle $R_{i_1}$.
	Thus, the only way for $\gamma(z)$ to be in the fundamental domain is to have $k = 0$; i.e., $\gamma$ is the identity.
	
	Next, we note that $\mathcal{S}$ acts discretely on hyperbolic space.
	This follows from Poincar\'e's Theorem on fundamental polygons (see \cite{beardon2012geometry} section 9.8).
	In Section \ref{sec:6cover}, we show that our Schottky groups have a finite cover by a reflection group.
	Since the angles in a fundamental polygon for this group can all be written as $\pi/n$ for some integer $n$ (see figure \ref{fd_reflection}), Poincar\'e's Theorem applies.
	Thus, as a finite-index subgroup of a discrete group acting on hyperbolic space, $\mathcal{S}$ also acts discretely.
	
	Finally, to see that every orbit has a point in the fundamental domain, one first checks that given a point in one of the circles $R_i$, reflection through that circle \textit{decreases} the distance from the point to $0$.
	Then, since $\mathcal{S}$ acts discretely, we can eventually move any point to the fundamental domain by repeatedly reflecting outside of any circle it falls in.
\end{proof}

\subsection{Rotational Symmetry}

We now consider the spectrum of the Laplacian on $L^2(\mathcal{S}\backslash\mathbb{D}$).
For any eigenvalue, it is often the case that a corresponding Maass form exists which exhibits a rotational symmetry.

To be more specific, let $\mathcal{S}$ be a symmetric Schottky group with parameter $\theta < 2\pi/3$, and suppose $\lambda$ is an eigenvalue of the Laplacian on $L^2(\mathcal{S}\backslash\mathbb{D})$.
It has been found experimentally that there tends to exist a Maass form for $\mathcal{S}$ with eigenvalue $\lambda$ which is invariant under rotation by $2\pi/3$.
We now give a heuristic argument for this observation.

Recall that we may map from the upper half plane model to the disk model via the Cayley transform
$$
z \mapsto \frac{z - i}{z + i}
$$
By a straightforward change of variables, we may convert the hyperbolic Laplacian from rectangular coordinates in the upper half plane model - see Equation (\ref{deltaUHP}) - to polar coordinates in the disk model.
This gives
$$
\Delta = -\left(\frac{(1 - \rho^2)^2}{4}\frac{\partial^2}{\partial\rho^2} +
\frac{(1 - \rho^2)^2}{4\rho}\frac{\partial}{\partial\rho} +
\frac{(1 - \rho^2)^2}{4\rho^2}\frac{\partial^2}{\partial\theta^2}\right)
$$
From this formula, one easily sees that rotations are $\Delta$-invariant.
Thus, if we compose any eigenfunction of the Laplacian with a rotation, we still have an eigenfunction of the same eigenvalue.

Next, note that a rotation of $2\pi/3$ is smooth on the fundamental domain for the $\mathcal{S}$.
Thus, given a Maass form $f$ for the eigenvalue $\lambda$, we can construct a $2\pi/3$ rotation-invariant Maass form of the same eigenvalue by taking the linear combination
$$
\frac{1}{3}\sum_{i = 0}^{2}f\left(e^{2\pi i/3} z\right)
$$
The only issue with the argument is that this linear combination may identically vanish, even if $f$ itself is not uniformly zero.
However, this at least gives an intuitive reason why we may expect to find a $2\pi/3$ rotation-invariant Maass form.

In any case, we see that there often is an underlying structure to these Schottky groups which is not immediately present in any flare expansion one would attempt to utilize in Hejhal's algorithm.
Instead of expecting the algorithm to come up with this symmetry on its own, we can build it into the problem by moving to a cover of our Schottky group.

\subsection{A 6-Fold Cover}\label{sec:6cover}

We start by considering a reflection group which contains our symmetric Schottky group $\mathcal{S}$ as a proper subset.
We will describe this group in the disk model.
We define three geodesics by stating their endpoints:
\begin{itemize}
	\item $D_1$ connects $-e^{i\pi/3}$ and $e^{i\pi/3}$
	\item $D_2$ connects -1 and 1
	\item $R$ connects $e^{-i\theta/2}$ and $e^{i\theta/2}$
\end{itemize}
where $\theta$ is the parameter for $\mathcal{S}$.
Let $\Gamma$ be the group generated by the reflections through these geodesics.
$$
\Gamma = \langle D_1, D_2, R \rangle
$$
(Note that we are abusing notation and letting $D_1, D_2$, and $R$ refer to the reflection through the geodesics as well as the geodesics themselves).
See Figure \ref{fd_reflection} for an example with $\theta = \pi/2$.
\begin{figure}
	\centering
	\includegraphics[width=0.6\linewidth]{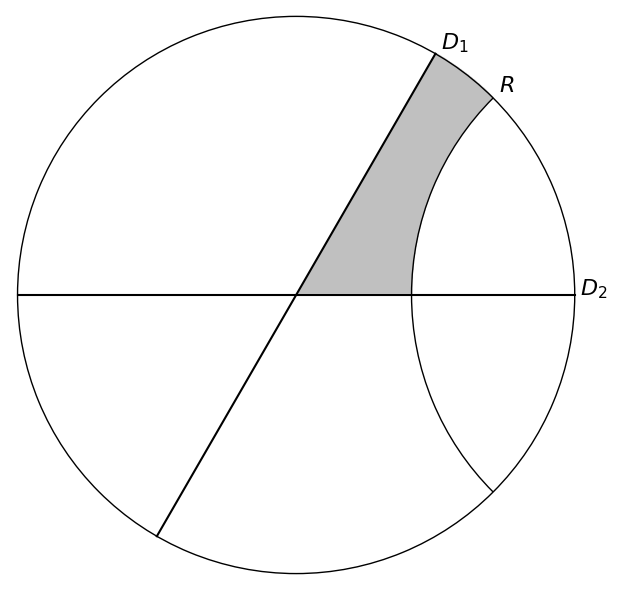}
	\caption{Fundamental Domain for $\Gamma$ with $\theta = \pi/2$}
	\label{fd_reflection}
\end{figure}

To see that $\Gamma$ contains $\mathcal{S}$, first observe that the composition $S = D_1D_2$ is just rotation by $2\pi/3$.
Thus, $S(R)$ and $S^2(R)$ give the other two circles in the Schottky group.
In particular, $S^{-1}RS$ and $SRS^{-1}$ give the reflections through those two circles.
Therefore, all three generators of $\mathcal{S}$ are contained in $\Gamma$.

In defense of the title of this section, we note that the index of $\mathcal{S}$ in $\Gamma$ is $[\Gamma : \mathcal{S}] = 6$.
This is most easily seen by noting that a fundamental domain for $\Gamma$ is one sixth of a fundamental domain for
$\mathcal{S}$.

\subsection{Doubling Across a Geodesic}

Since $\Gamma$ contains $\mathcal{S}$, any Maass form for $\Gamma$ is automatically a Maass form for $\mathcal{S}$.
In fact, it is possible to show that $\Gamma$ has the same bottom eigenvalue as the symmetric Schottky group, and the corresponding base eigenfunction for $\mathcal{S}$ does indeed have the required symmetries to be a Maass form on $\Gamma$.
We note however that both $\Gamma$ and $\mathcal{S}$ are reflection groups, and so they contain elements which are not M\"obius transformations.
If one wishes to work solely with orientation-preserving isometries, we can make use of a technique called \textit{doubling}.

As an example, we will double the group $\Gamma$ across $D_1$.
Specifically, consider
$$
\Gamma_0 = \langle D_1D_2, D_1R \rangle
$$
The doubled group then consists solely of M\"obius transformations, since they are the result of the composition of two reflections.
This is called \textit{doubling} across $D_1$ since the fundamental domain for the doubled group is exactly the old fundamental domain plus its reflection across $D_1$ (see Figure \ref{fd_reflection_doubled}).
\begin{figure}
	\centering
	\includegraphics[width=0.6\linewidth]{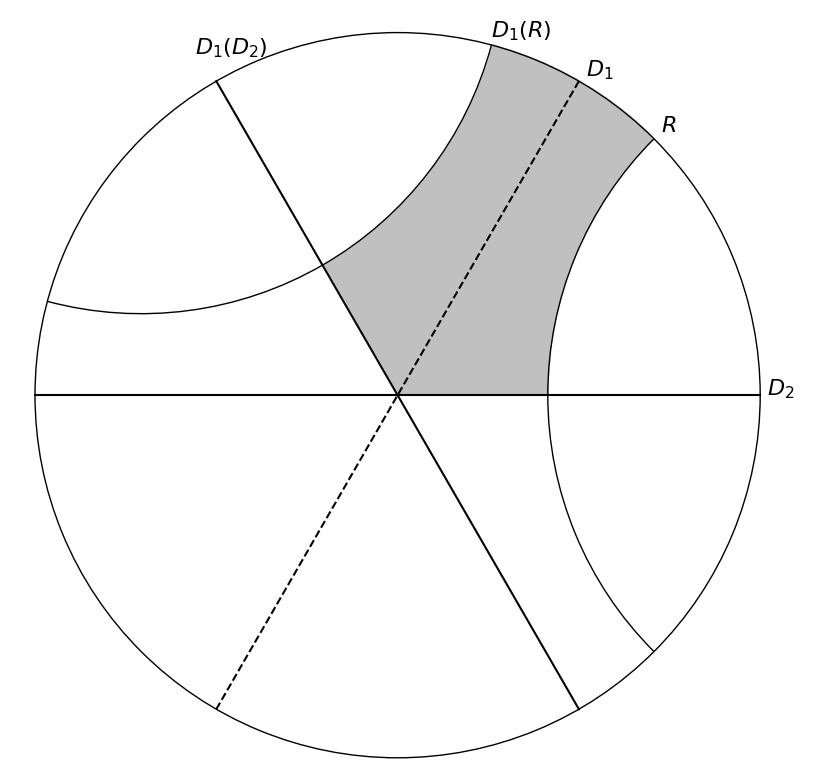}
	\caption{Fundamental Domain for $\Gamma_0$ with $\theta = \pi/2$}
	\label{fd_reflection_doubled}
\end{figure}

Next, we will show how $\Gamma_0$ can be realized as a matrix group.

\begin{lem}
	$\Gamma_0$ is isomorphic to the matrix group
	$$
	\left\langle
	\begin{pmatrix}
		e^{i\pi/3} & ~ \\
		~ & e^{-i\pi/3}
	\end{pmatrix},
	\begin{pmatrix}
		i\csc\frac{\theta}{2} & -i\cot\frac{\theta}{2} \\
		i\cot\frac{\theta}{2} & -i\csc\frac{\theta}{2}
	\end{pmatrix}
	\right\rangle < \textup{PSU}(1, 1)
	$$
\end{lem}

\begin{proof}
	First note that the composition $D_1D_2$ is just rotation by angle $2\pi/3$.
	This gives the first matrix in the lemma statement.
	Next, we replace $D_1R$ with $(D_1D_2)^{-1}D_1R = D_2R$ as a generator.
	This is helpful since the action of $D_2$ is simply complex conjugation.
	Since the circle defined by the geodesic $R$ has center $\sec(\theta/2)$ and radius $\tan(\theta/2)$, we find that reflection through $R$ can be written as
	$$
	R(z) = \frac{\tan^2(\theta/2)}{\bar{z} - \sec(\theta/2)} + \sec(\theta/2)
	$$
	We can now explicitly write down the action of $D_2R$.
	$$
	D_2R(z) = \overline{R(z)} =
	\frac{\sec(\theta/2)z - 1}{z - \sec(\theta/2)}
	$$
	Finally, note that multiplying the numerator and denominator by $i\cot(\theta/2)$ converts this to a M\"obius transformation where the corresponding matrix is in $\text{PSU}(1, 1)$ (this is the matrix group which is conjugate via the Cayley transform to $\text{PSL}(2, \mathbb{R})$).
	Specifically,
	$$
	D_2R(z) = \frac{i\csc(\theta/2)z - i\cot(\theta/2)}{i\cot(\theta/2)z - i\csc(\theta/2)}
	$$
	This action is given by the second matrix in the lemma statement, and so the proof is finished.
\end{proof}

Let us briefly note that it is not strictly necessary to move from $\Gamma$ to $\Gamma_0$ in order to apply Hejhal's algorithm.
A Maass form on a group is automatically a Maass form for any of its subgroups, but the converse of this statement is not true.
By searching for Maass forms of $\Gamma_0$ and $\Gamma$, we don't guarantee that we find all forms for $\mathcal{S}$ (in fact, note that Maass forms for $\Gamma_0$ are not necessarily invariant under the action of $\mathcal{S}$).
However, it was found experimentally that these three groups indeed share many Maass forms (see Section \ref{ch:Results} for methods to verify that a given set of Fourier coefficients define a Maass form).
Moreover, while Hejhal's algorithm failed to converge when using the original group $\mathcal{S}$, it succeeded when utilizing the extra structure provided by $\Gamma$ and $\Gamma_0$.

\subsection{A Flare Domain for $\Gamma_0$}

So far in this section, we have worked solely in the disk model of hyperbolic space.
However, recall that our flare domains are defined in polar coordinates in the upper half plane model.
In this section, we show how to move $\Gamma_0$ to a flare domain. 

First, recall that we can map between the upper half plane model and the disk model via the \textit{Cayley transform} and its inverse
$$
C(z) = \frac{z - i}{z + i} ~~~~~~~~ C^{-1}(w) = i\cdot\frac{1 + w}{1 - w}
$$
Moreover, it can be shown that $C$ defines an isometry between the unit tangent bundles $T^1\mathbb{H}$ and $T^1\mathbb{D}$.

Next, recall that the orientation-preserving isometries on $\mathbb{H}$ can be described as the set of M\"obius transformations in $\text{PSL}(2, \mathbb{R})$ (this is discussed in more detail in section \ref{sec:FuchsianGroups}).
For $g \in \text{PSL}(2, \mathbb{R})$ acting on $\mathbb{H}$, the equivalent action on $\mathbb{D}$ is $CgC^{-1}$.
One can show that
$$
C\text{PSL}(2, \mathbb{R})C^{-1} = \text{PSU}(1, 1) :=
\left\{ 
\begin{pmatrix}
	\alpha & \beta \\
	\bar{\beta} & \bar{\alpha}
\end{pmatrix} \in M(2, \mathbb{C}) : |\alpha|^2 - |\beta|^2 = 1
\right\} / \{\pm I\}
$$
where we are thinking of $C$ as the matrix $\begin{psmallmatrix} 1 & -i \\ 1 & i \end{psmallmatrix} \in \text{PSL}(2, \mathbb{C})$, i.e., the usual identification of M\"obius transformations with matrices.

We are now ready to map the matrix group $\Gamma_0$ to the upper half plane model.
One finds that
$$
C^{-1}D_1D_2C =
\begin{pmatrix}
	\frac{1}{2} & \frac{\sqrt{3}}{2} \\
	-\frac{\sqrt{3}}{2} & \frac{1}{2}
\end{pmatrix}
~~~~~~~~
C^{-1}D_2RC =
\begin{pmatrix}
	~ & \cot\frac{\theta}{4} \\
	-\tan\frac{\theta}{4} & ~
\end{pmatrix}
$$
and so these are the generators in $\text{PSL}(2, \mathbb{R})$.
To get a fundamental domain in $\mathbb{H}$, we simply take the inverse Cayley transform of the fundamental domain in $\mathbb{D}$ (see Figure \ref{fd_doubled_UHP}).
\begin{figure}
	\centering
	\includegraphics[width=\linewidth]{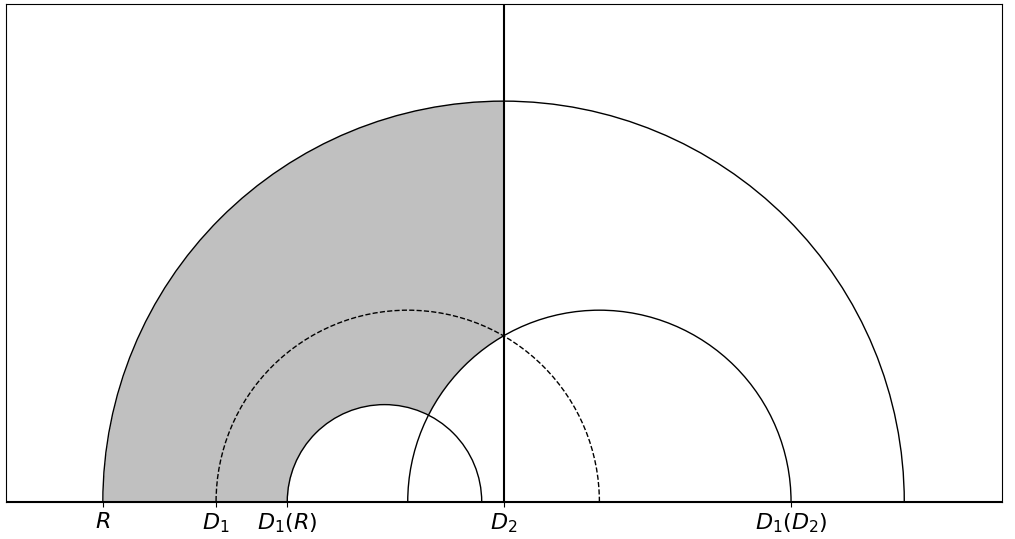}
	\caption{Fundamental domain in the upper half plane model for $C^{-1}\Gamma_0C$ with $\theta = \pi/2$}
	\label{fd_doubled_UHP}
\end{figure}

Next, we want to use the flare present in $\Gamma_0$'s fundamental domain to map over to a flare domain.
The first step is to identify the matrix in $\Gamma_0$ whose axis cuts off the flare (see Section \ref{sec:FlareDef}).
In the case of $\Gamma_0$, one finds that $D_1R$ is the appropriate transformation (see Figures \ref{cover_with_flare} and \ref{cover_with_flare_UHP} for geometric verification).
\begin{figure}
	\centering
	\includegraphics[width=0.7\linewidth]{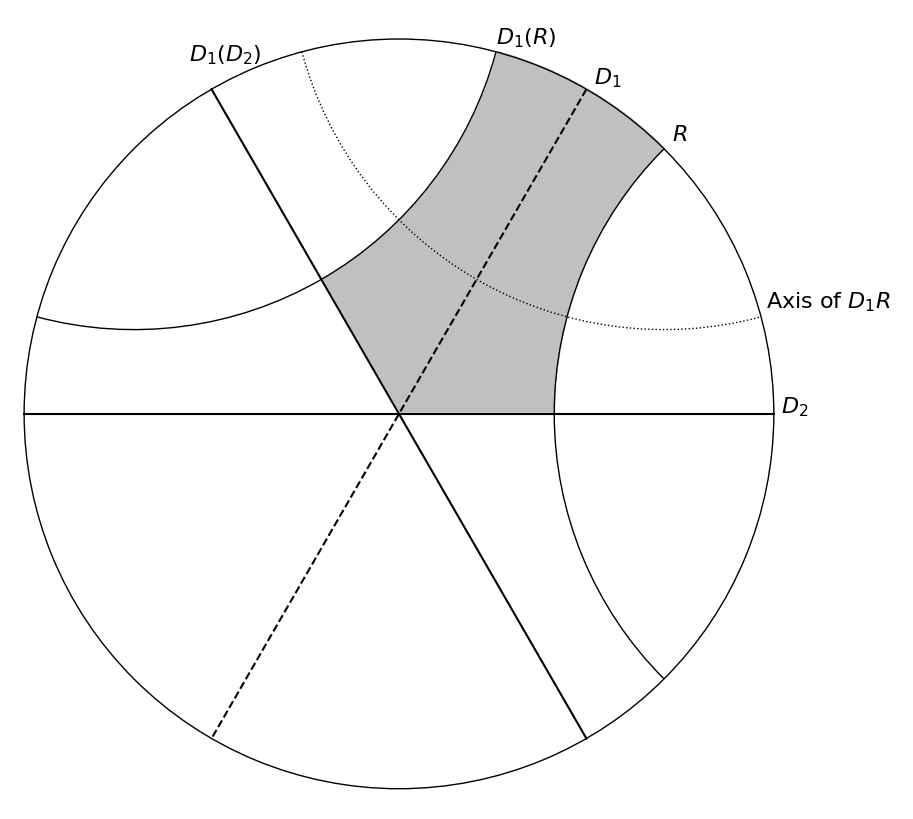}
	\caption{Axis of $D_1R$ cuts off the flare}
	\label{cover_with_flare}
\end{figure}
\begin{figure}
	\centering
	\includegraphics[width=\linewidth]{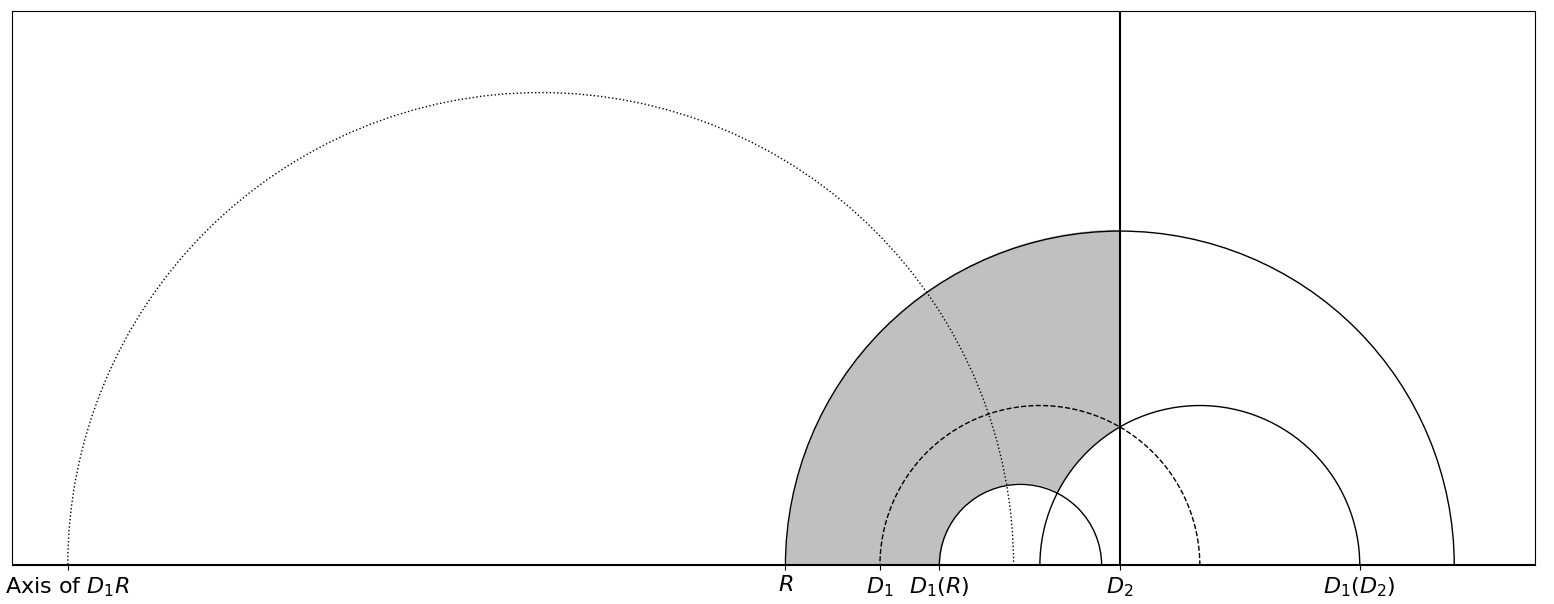}
	\caption{Axis of $D_1R$ cuts off the flare}
	\label{cover_with_flare_UHP}
\end{figure}

In matrix form, one finds that
$$
D_1R =
\frac{1}{2}
\begin{pmatrix}
	-\sqrt{3} + i\csc\frac{\theta}{2} & \sqrt{3} - i\cot\frac{\theta}{2} \\
	\sqrt{3} + i\cot\frac{\theta}{2} & -\sqrt{3} - i\csc\frac{\theta}{2}
\end{pmatrix}
$$
which, in the upper half plane model, is
$$
C^{-1}D_1RC =
\frac{1}{2}
\begin{pmatrix}
	-\sqrt{3}\tan\frac{\theta}{4} & \cot\frac{\theta}{4} \\
	-\tan\frac{\theta}{4} & -\sqrt{3}\cot\frac{\theta}{4}
\end{pmatrix}
$$

In particular, we can utilize the flare in the upper half plane model to map over to a flare domain (again, see Section \ref{sec:FlareDef} for a discussion of this map).
Letting $z_1 < z_2$ be the endpoints of the axis defined by $C^{-1}D_1RC$ (in the upper half plane model), and letting $t$ be the left end point of $R$ (again, in $\mathbb{H}$), we apply the map
$$
U(z) = \left( \frac{t - z_2}{t - z_1} \right) \frac{z - z_1}{z - z_2}
$$
to our fundamental domain and end up with the flare domain shown in Figure \ref{cover_flare}.
\begin{figure}
	\centering
	\includegraphics[width=\linewidth]{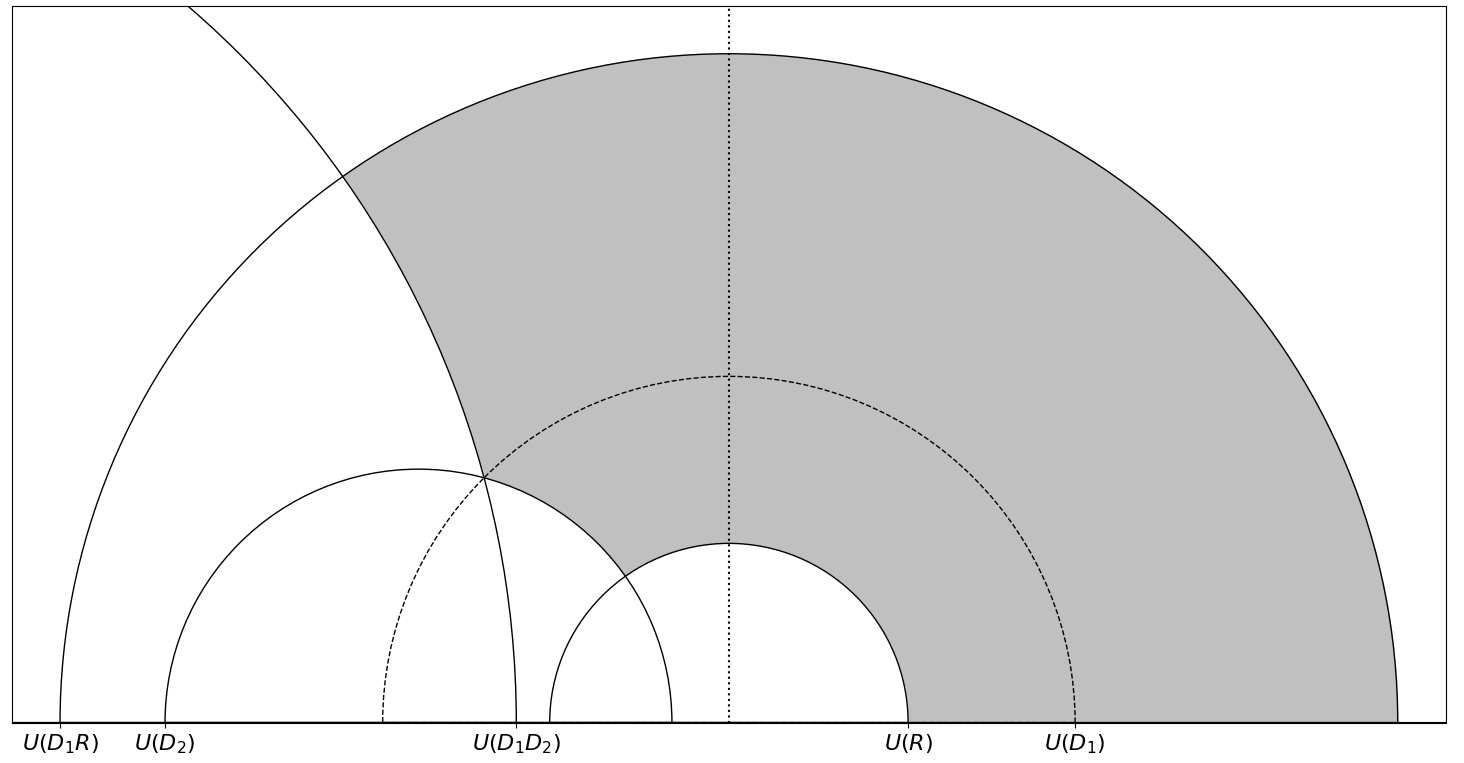}
	\caption{Flare domain for the group $\Gamma_0$ with parameter $\theta = \frac{\pi}{2}$}
	\label{cover_flare}
\end{figure}
It is this flare domain which exhibits the logarithmic Fourier expansion needed for our extension of Hejhal's algorithm.


\section{Hejhal's Algorithm applied to Infinite Volume Hecke Groups}\label{ch:Hecke}

In this section, we apply the new algorithm to groups which contain both a flare \textit{and a cusp}.
As discussed in section \ref{sec:MultiExpansion}, this involves solving for the Fourier coefficients of different expansions simultaneously.
Heuristically, utilizing multiple expansions should aid in the convergence of the algorithm.
This is because each expansion has its own built-in invariance; cuspidal expansions exhibit invariance under certain parabolic elements of the group, and flare expansions exhibit invariance under specific hyperbolic transformations.
By forcing our function to exhibit both such behaviors, we severely limit the search space in which to find a Maass form.

Our example for this section is that of infinite volume Hecke groups, which we defined in section \ref{sec:IntroHecke}.
We will show how to move from the standard fundamental domain to a flare domain, then we move to a discussion on choosing test points for Hejhal's algorithm.
When multiple expansions are involved, the choice of test points becomes even more important to the algorithm.
In particular, each expansion must have admissible test points; otherwise, we have no information about the Fourier coefficients for that expansion.

\subsection{Fundamental Domain for Hecke Groups}

First, we describe a fundamental domain for an infinite volume Hecke group which we will use to construct a flare domain.

\begin{lem}\label{fund_Hecke}
	Let $\Gamma_r$ be the Fuchsian group generated by $z \mapsto z + 1$ and $z \mapsto -r^2/z$ with $0 < r < 1/2$.
	Then a fundamental domain for $\Gamma_r$ is given by
	$$
	\mathcal{F}_r = \{ z = x + iy \in \mathbb{H} : 0 < x < 1, |z| > r, |z - 1| < r \}
	$$
\end{lem}

\begin{proof}
	First, we argue that an infinite volume Hecke group acts discretely on hyperbolic space.
	As in the case of the symmetric Schottky groups, this follows from Poincar\'e's Theorem on fundamental polygons (see \cite{beardon2012geometry} section 9.8).
	Consider the reflection group $\mathcal{R}$ generated by reflections through the geodesics $g_0 = \{iy \in \mathbb{H} \}, g_1 = \{1/2 + iy \in \mathbb{H}\},$ and $g_r = \{z \in \mathbb{H} : |z| = r\}$.
	One fundamental polygon for this is the region bounded by the three geodesics.
	The only vertex of this polygon occurs at $\infty$ between $g_0$ and $g_1$, and the angle there is $\pi/\infty$.
	Thus, Poincar\'e's Theorem applies.
	On the other hand, one may check directly from the generators that $\Gamma_r$ is an index-2 subgroup of $\mathcal{R}$.
	So, $\Gamma_r$ also acts discretely on hyperbolic space.
	
	Next, we show that the set
	$$
	\mathcal{F} := \left\{ z = x + iy : -\frac{1}{2} < x < \frac{1}{2}, |z| > r \right\}
	$$
	is a fundamental domain for $\Gamma_r$ (see figure \ref{Hecke_orig}).
	Let $z_0 \in \mathbb{H}$ be arbitrary, and consider the following sequence of moves
	\begin{enumerate}
		\item By repeated application of $z \mapsto z + 1$ (or its inverse $z \mapsto z - 1$), move $z_0$ to a point $z^*$ with $-1/2 \leq \text{Re}(z^*) \leq 1/2$
		\item If $|z^*| < r$, apply $z \mapsto -r^2/z$
		\item Repeat steps 1 and 2
	\end{enumerate}
	If at move 2 the point $z^*$ has $|z^*| \geq r$, then we have found a point in the closure of $\mathcal{F}$ which is equivalent under $\Gamma_r$ to $z_0$.
	On the other hand, note that move 1 does not affect the imaginary part of of $z_0$, and move 2 strictly increases the imaginary part of $z^*$.
	If this sequence of moves were to repeat indefinitely, then we would find an infinite number of distinct points in the compact region
	$$
	\{ z \in \mathbb{H} : \text{Im}(z) \geq \text{Im}(z_0), |z| \leq |r| \}
	$$
	which are equivalent under $\Gamma_r$.
	This would contradict that $\Gamma_r$ acts discretely on $\mathbb{H}$, and therefore the steps will terminate in finitely many moves.
	
	The argument above shows that any $z \in \mathbb{H}$ is equivalent under the action of $\Gamma_r$ to a point in $\overline{\mathcal{F}}$.
	Thus, $\mathcal{F}$ contains a fundamental domain for $\Gamma_r$.
	Conversely, one easily checks that the hyperbolic volume of $\mathcal{F}$ is exactly twice the volume of a fundamental polygon for the reflection group $\mathcal{R}$ described above.
	Since $[\mathcal{R} : \Gamma_r] = 2$, a fundamental domain for $\Gamma_r$ cannot be strictly smaller than $\mathcal{F}$.
	This completes the proof that $\mathcal{F}$ is a fundamental domain for $\Gamma_r$.
	
	To finish, we simply note that applying $z \mapsto z + 1$ to the set $\{ x + iy \in \mathcal{F} : x < 0 \}$ takes $\mathcal{F}$ to $\mathcal{F}_r$ (up to a set of measure zero).
	\begin{figure}
		\centering
		\begin{subfigure}{.5\textwidth}
			\centering
			\includegraphics[width=.9\linewidth]{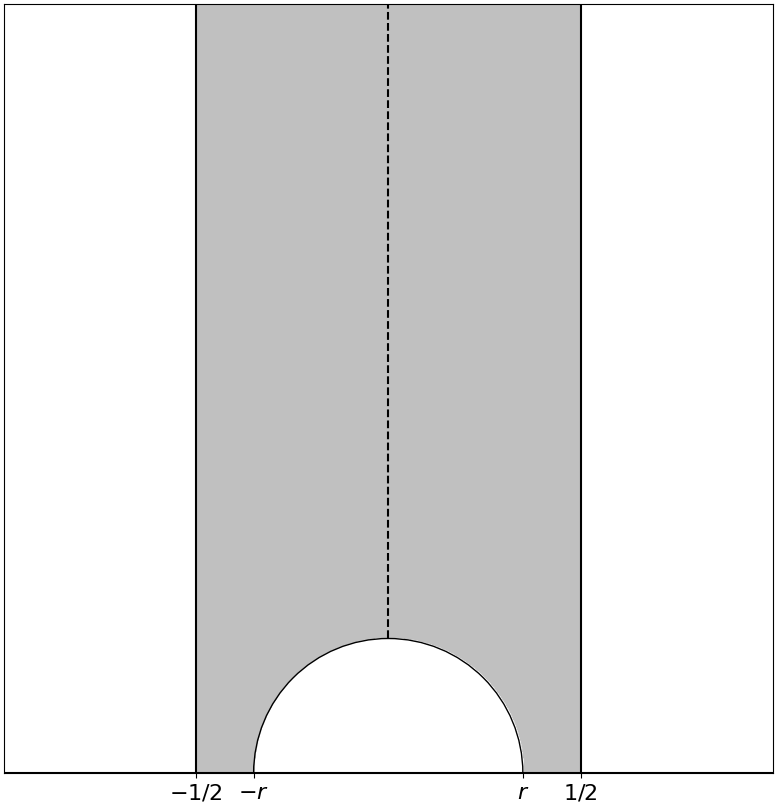}
			\caption{$\mathcal{F}$}
			\label{equiv_domains:suba}
		\end{subfigure}%
		\begin{subfigure}{.5\textwidth}
			\centering
			\includegraphics[width=.9\linewidth]{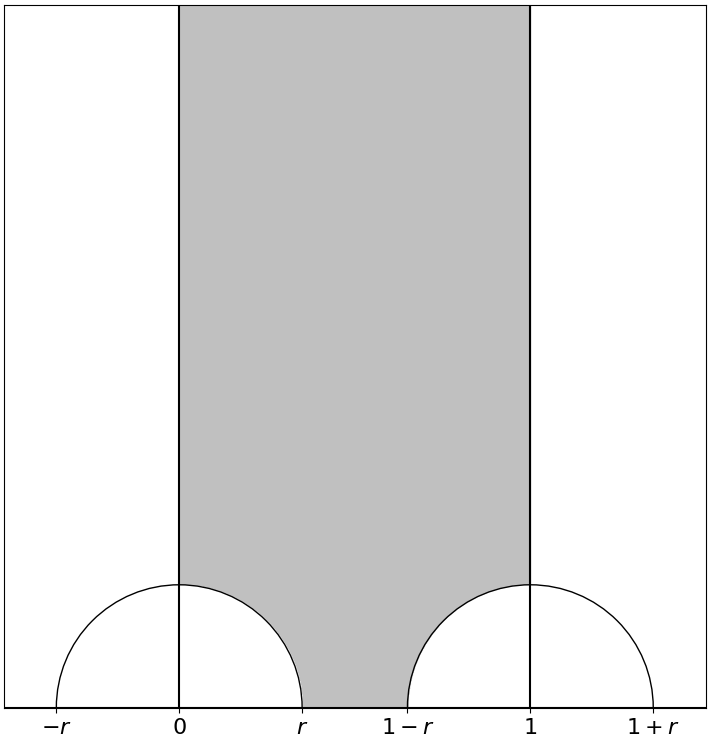}
			\caption{$\mathcal{F}_r$}
			\label{equiv_domains:subb}
		\end{subfigure}
		\caption{Translating the left half of $\mathcal{F}$ by 1 maps onto $\mathcal{F}_r$.}
		\label{equiv_domains}
	\end{figure}
	See Figure \ref{equiv_domains} to view the two fundamental domains side-by-side.
	
\end{proof}

We prefer the fundamental domain $\mathcal{F}_r$ to $\mathcal{F}$ because it is more clear that the group exhibits only one flare.
In Figure \ref{equiv_domains:suba}, it appears at first glance that there are two separate infinite volume portions of the fundamental domain.
However, since the $-1/2$ and $1/2$ lines are identified via $z \mapsto z + 1$, this is in fact just one continuous region bordering the real axis.

Let us also note that the proof above provides a pullback algorithm for $\mathcal{F}_r$.
Given any point $z \in \mathbb{H}$, we saw an algorithm to find a point $z^* \in \mathcal{F}$ which is equivalent to $z$ under $\Gamma_r$.
If $\text{Re}(z^*) \leq 1/2$, then $z^* + 1 \in \mathcal{F}_r$.
Otherwise, $z^*$ is already in $\mathcal{F}_r$.

\subsection{A Flare Domain for $\Gamma_r$}

Identifying M\"obius transformations with $\text{PSL}(2, \mathbb{R})$ in the usual manner, we can write our Hecke group in terms of matrices.
$$
\Gamma_r = \left\langle
\begin{pmatrix}
	1 & 1 \\
	0 & 1
\end{pmatrix},
\begin{pmatrix}
	0 & -r \\
	1/r & 0
\end{pmatrix}
\right\rangle
$$
where $0 < r < 1/2$.
To map to a flare domain, we require a hyperbolic matrix in $\Gamma_r$.
Ideally, the matrix we use would have an axis which cuts off the flare in our chosen fundamental domain.

Define the geodesics $g_1 = \{z \in \mathbb{H} : |z| = r\}$ and $g_2 = \{z \in \mathbb{H} : |z - 1| = r\}$.
Since $z \mapsto -r^2/z$ reflects $g_1$ across the imaginary axis and $z \mapsto z + 1$ maps $g_1$ to $g_2$, we note that the composition of these maps identifies one side of the flare with the other.
In other words, we are interested in the following matrix
$$
A :=
\begin{pmatrix}
	1 & 1 \\
	0 & 1
\end{pmatrix}
\begin{pmatrix}
	0 & -r \\
	1/r & 0
\end{pmatrix} =
\begin{pmatrix}
	1/r & -r \\
	1/r & 0
\end{pmatrix}
$$
Since $0 < r < 1/2$, we have that tr$(A) = 1/r > 2$, so $A$ is indeed hyperbolic.
We claim that diagonalizing this matrix results in a flare domain for $\Gamma_r$.

\begin{lem}\label{hecke_flare_lemma}
	Define $A$ as above, and let $C$ be $A$'s axis (i.e. $C$ is the geodesic between the fixed points of $A$).
	Let $z_1 < z_2$ be the endpoints of $A$ in $\mathbb{R}$, and let $U$ be the matrix whose action on $\mathbb{H}$ is
	$$
	U(z) = c\frac{z - z_1}{z - z_2}
	$$
	where $c = (r - z_2)/(r - z_1)$.
	Then $C$ cuts off a flare in $\mathcal{F}_r$, $UAU^{-1}$ is diagonal, and $U\Gamma_rU^{-1}$ acting on $U\mathbb{H}$ exhibits a flare domain.
	Furthermore, $U(\mathcal{F}_r)$ is itself already a flare domain.
\end{lem}

\begin{proof}
	A straightforward computation gives the fixed points of $A$.
	$$
	z_1 = \frac{1 - \sqrt{1 - 4r^2}}{2} ~~~~ z_2 = \frac{1 + \sqrt{1 - 4r^2}}{2}
	$$
	Moreover, one can show that $z_1 < r < 1 - r < z_2$ and that $C$ is orthogonal to $g_1$ and $g_2$.
	This is enough to conclude that $C$ cuts off the flare for $\mathcal{F}_r$; see Figure \ref{FrWithCs} for a visual.
	\begin{figure}[h]
		\centering
		\includegraphics[width=0.8\linewidth]{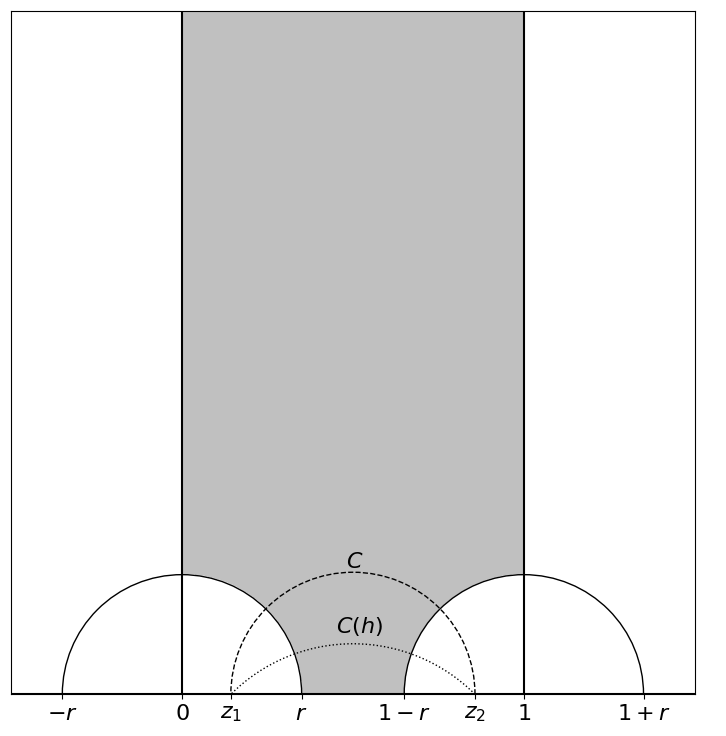}
		\caption{$\mathcal{F}_r$ for $r = 7/20$ with $C$ and $C(h)$ for some $h > 0$}
		\label{FrWithCs}
	\end{figure}
	
	Next, we argue that $UAU^{-1}$ is diagonal.
	One could of course check this directly, but we proceed by a more geometric argument which gives greater insight.
	First, notice that $U$ is chosen so that $U(z_1) = 0$ and $U(z_2) = \infty$.
	This gives us that $UAU^{-1}(0) = 0$ and $UAU^{-1}(\infty) = \infty$.
	The elements of $SL(2, \mathbb{R})$ which stabilize infinity have bottom-left entry equal to $0$, so we conclude that $UAU^{-1}$ has the form
	$\begin{psmallmatrix}
		a & b \\
		0 & d
	\end{psmallmatrix}$
	where $ad = 1$.
	Since we also have that $UAU^{-1}(0) = b/d$, we conclude that $b = 0$.
	So this matrix is diagonal, as claimed.
	
	Thus far, we have found a hyperbolic matrix $A \in \Gamma_r$ and a matrix $U \in SL(2, \mathbb{R})$ which diagonalizes it.
	Moreover, since $c$ is chosen so that $U(r) = 1$ and $\kappa := U(1 - r) > 1$, we have that \textit{some} fundamental domain for the action of $U\Gamma_rU^{-1}$ on $U\mathbb{H}$ will be a flare domain.
	By simply applying $U$ to any point in $\mathcal{F}_r$ and noting that it has norm between 1 and $\kappa$, we can confirm that $U(\mathcal{F}_r)$ is a flare domain.
	This is indeed the case, as seen in Figure \ref{UFr}.
	\begin{figure}[h]
		\centering
		\includegraphics[width=\linewidth]{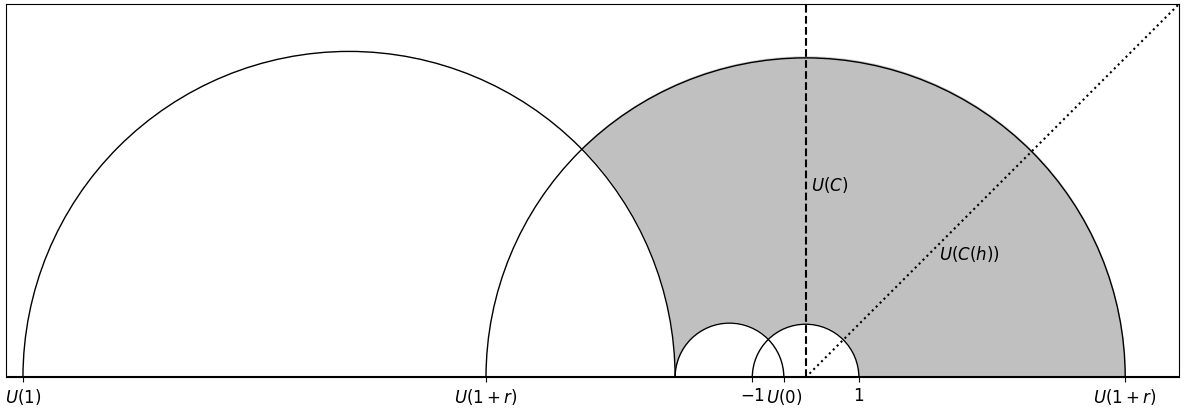}
		\caption{$U(\mathcal{F}_r)$ with $\mathcal{F}_r$ as in Figure \ref{FrWithCs}}
		\label{UFr}
	\end{figure}
\end{proof}

In Section \ref{sec:IntroFlares}, we parameterized a flare both by its width $\kappa$ and its cutoff angle $\alpha$.
To picture the cutoff angle in this context, we first define the curve $C(h)$ to be the arc of the circle centered at $\frac{1}{2} - ih$ ($h > 0$) which passes through $z_1$ and $z_2$.
Using classical geometry, it can be shown that the acute angle $C(h)$ makes with the real axis is
$$
\alpha = \frac{\pi}{2} - \arctan\left( \frac{h}{\sqrt{\frac{1}{4} - r^2}} \right)
$$
Next, since M\"obius transformations are conformal maps, the angle $C(h)$ makes with the real axis in the original domain is the same as the angle $U(C(h))$ makes.
So the angle $\alpha$ defined above becomes exactly the cutoff angle of the flare.
Note that if we have a specific value of $\alpha$ in mind, it would be straightforward to solve for $h > 0$ in this equation.
Moreover, we saw in the proof of Lemma \ref{hecke_flare_lemma} that the scaling parameter $\kappa$ is also easily described.
$$
\kappa = U(1-r) = \left(\frac{z_2}{r}\right)^2
$$
Since trace is conjugation invariant, we could also compute $\kappa$ by noting that
$$
\frac{1}{r} = \text{tr}(A) = \text{tr}(UAU^{-1}) = \sqrt{\kappa} + \frac{1}{\sqrt{\kappa}}
$$

\subsection{Choosing Test Points for Multiple Expansions}

In Section \ref{sec:TestPointsCusp}, we discussed choosing test points with respect to a cuspidal expansion.
In Section \ref{sec:TestPointsFlare}, we did the same for flare expansions.
The infinite volume Hecke groups discussed in this section exhibit both expansions.
To handle this, we simply choose a set of test points with respect to each expansion; our full set of test points is then the concatenation of these sets.

In Figure \ref{test_points_orig}, we display the original fundamental domain which exhibits a width 1 cusp at infinity.
\begin{figure}
	\centering
	\includegraphics[width=.7\linewidth]{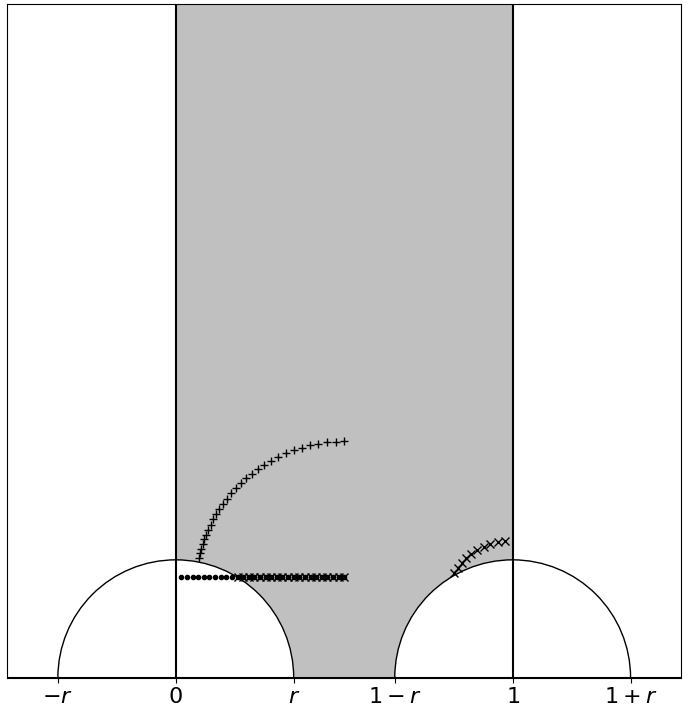}
	\caption{Test points in the original fundamental domain. A horocycle of height $3/10$ is shown as a set of points, and their pullbacks are shown as $\times$'s. The points from a ray of angle $\theta = 2.2$ in the flare expansion are shown as $+$'s.}
	\label{test_points_orig}
\end{figure}
Choosing test points with respect to the cuspidal expansion amounts to taking points along a low-lying horocycle.\footnote{we take half the horocycle because infinite volume Hecke groups admit a double cover by a reflection group, as seen in the proof of Lemma \ref{fund_Hecke}; this group includes reflection across the geodesic $\{\text{Re}(z) = 1/2\}$}
We choose the horocycle so that some test points fall outside the fundamental domain and some fall inside.
Recall from Section \ref{sec:MultiExpansion} that in the presence of multiple expansions, test points may be taken inside the fundamental domain.
These provide equations for our linear system by equating two different Fourier expansions at each such point.

We also note in Figure \ref{test_points_orig} that all test points - both those coming from a low-lying horocycle and those from a ray - are bounded away from the real line.
This is necessary for all points to be admissible with respect to the cuspidal expansion.

In Figure \ref{test_points_flare}, we show the same fundamental domain and test points after mapping to a flare.
\begin{figure}
	\centering
	\includegraphics[width=\linewidth]{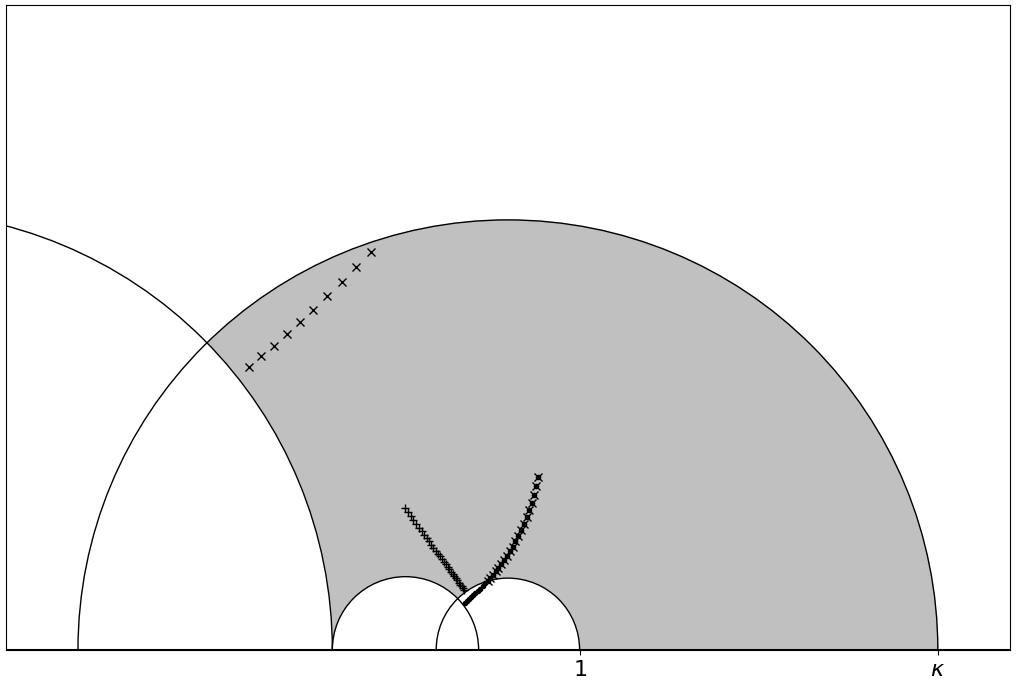}
	\caption{This is the same as Figure \ref{test_points_orig} after mapping to a flare domain.}
	\label{test_points_flare}
\end{figure}
The test points chosen with respect to the flare domain are along a ray of fixed angle (we stop at $\sqrt{\kappa}$ instead of $\kappa$ for the same reason we took half a horocycle: the double cover by a reflection group).
Both these points and those coming from the low-lying horocycle in the original domain are bounded away from an argument of $\pi$, so again all points are admissible with respect to the flare expansion.

In summary, we chose test points with respect to both expansions and both inside and outside of the fundamental domain.
Such variety ensures that rich group information is present in the linear system which appears in Hejhal's algorithm, thereby aiding in convergence towards a true Maass form.


\section{Results}\label{ch:Results}

In this section, we present results from experiments on Schottky groups and infinite volume Hecke groups.
A motivator for choosing these specific examples is that they were studied by McMullen \cite{mcmullen1998hausdorff}.
(Similar results are obtained by Jenkinson and Pollicott \cite{jenkinson2002calculating} and more recently by Pollicott and Vytnova \cite{pollicott2020hausdorff}).
In his paper, McMullen presents an algorithm for computing the Hausdorff dimension of the limit sets of certain Kleinian groups and Julia sets of rational maps.
As discussed in the introduction, the Hausdorff dimension $\delta$ is related to the bottom eigenvalue $\lambda_0$ of the Laplacian by the formula
\begin{equation}\label{delta_lambda}
	\lambda_0 = \delta(1 - \delta)
\end{equation}
Thus, McMullen's algorithm computes both $\delta$ and $\lambda_0$.

The extension of Hejhal's algorithm presented in this paper has both advantages and disadvantages when compared to McMullen's algorithm.
One advantage of McMullen's work is its applicability to cases in which the group in question does not admit Maass forms.
Equation \ref{delta_lambda} presupposes the existence of a base eigenfunction; we will see this issue come up below when discussing results for Schottky groups.

We also note here a few distinct advantages of using Hejhal's algorithm.
First, it computes more information about the group.
Our algorithm estimates the Fourier coefficients of the Maass form (thus estimating the Maass form itself), and it can be used to compute \textit{all} eigenvalues of the Laplacian - not just the bottom eigenvalue.
Another advantage of Hejhal's algorithm over McMullen's is efficiency.
In cases where both algorithms apply, our extension of Hejhal's algorithm computes the Hausdorff dimension to a much higher accuracy without sacrificing speed (see high precision examples in the GitHub code).

\subsection{Verifying Results}

In \cite{booker2006effective}, Booker, Str\"ombergsson, and Venkatesh consider the effective computation and \textit{verifaction} of Maass forms for $\text{SL}(2, \mathbb{Z})$ using Hejhal's algorithm.
They further consider some difficulties in extending their results to congruence subgroups.
For our infinite volume examples, there is an even larger gap in the knowledge required to obtain a result like theirs.
For example, most of the groups we consider are non-algebraic, so there are no Hecke relations which may be checked on the Fourier coefficients.

On the other hand, there is a plethora of heuristic evidence which may be used to verify our results.
The simplest example of this is that the algorithm utilizes least squares on an \textit{overdetermined} system.
A random overdetermined system has zero probability of having a solution, so whenever our algorithm converges at all, we can be reasonably certain that the convergence is towards a true eigenvalue-eigenfunction pair.

Moreover, in specific examples, we may look at facts about the Maass form which were not used to set up Hejhal's algorithm and check that they still hold for the function we computed.
For example, we saw in Lemma \ref{nonneg} that the base eigenfunction can be taken to be nonnegative.
However, there is nothing in our version of Hejhal's algorithm which forces the resulting function to have this property.
If the approximated Maass form ends up being nonnegative, that is another heuristic which strengthens our confidence in the result.

For the rest of this section, we look at results for groups that were studied by McMullen.
He used a totally different algorithm to compute the Hausdorff dimension of limit sets.
Therefore, wherever our results match his, we gain confidence in the accuracy of our own method.

\subsection{Results on Schottky Groups}\label{sec:resultsSchottky}

In \cite{mcmullen1998hausdorff} Table 12, McMullen lists the Hausdorff dimension of the limit set of every symmetric Schottky group with angles $\theta = 1, 2, 3, \dots, 120$ (in degrees).
To utilize our algorithm for the same purpose, we require the existence of Maass forms on the group.
A Lax-Phillips argument shows that $L^2$ functions can only occur for Hausdorff dimensions greater than $1/2$. Therefore, we only test on values of $\theta$ greater than or equal to $94$ degrees; this is the cutoff point for which the dimension exceeds $1/2$, as seen in McMullen's table.

Our own results for three symmetric Schottky groups are listed below to 10 decimal points of accuracy (see Appendix \ref{ch:Appx_Results} for the full table).
Indeed, the Hausdorff dimensions match McMullen's to 8 decimal points (this being the accuracy given in his paper).
\begin{table}[h!]
	\centering
	\begin{tabular}{|c|c|c|c|}
	\hline
	$\theta$ & 94 & 95 & 96 \\
	\hline
	$\delta$ & 0.5063972405 & 0.5155835572 & 0.5250520005 \\
	\hline
	$b_0$ & 1.0 & 1.0 & 1.0 \\
	\hline
	$b_1$ & -2.5166508538E-7 & -1.7300804172E-7 & -1.1617826121E-7 \\
	\hline
	$b_2$ & 1.1198755812E-14 & 4.8519808552E-15 & 1.9929273692E-15 \\
	\hline
	$b_3$ & -2.0726623244E-22 & -4.2399625678E-23 & -6.1816973015E-24 \\
	\hline
	$b_4$ & -2.3468239788E-29 & -5.5134294682E-30 & -1.141094769E-30 \\
	\hline
	$b_5$ & 3.4777754033E-36 & 4.8036304002E-37 & 5.505497463E-38 \\
	\hline
	$b_6$ & -2.7250500511E-43 & -2.7978189536E-44 & -1.7376105471E-45 \\
	\hline
	$b_7$ & -6.1842978186E-51 & -7.4067873863E-51 & 5.4124177087E-53 \\
	\hline
	$b_8$ & -2.3465059548E-57 & -4.0643898947E-57 & 2.1949093323E-59 \\
	\hline
	$b_9$ & 4.5291057657E-64 & 1.7973473077E-64 & 9.3505260894E-66 \\
	\hline
	\end{tabular}
	
	\caption{Hausdorff dimensions and first 10 Fourier coefficients of the base eigenfunction for various symmetric Schottky groups. The eigenfunctions are scaled so that the $0^{th}$ coefficient is 1.}
	\label{schottky_table}
	
\end{table}

\subsection{Results on Hecke Groups}

In \cite{mcmullen1998hausdorff}, McMullen also discusses the Hausdorff dimension of the limit sets of infinite volume Hecke groups.
Just as in the case of symmetric Schottky groups, we are only able to use our extension of Hejhal's algorithm in cases where the dimension exceeds $1/2$.
On the other hand, our algorithm simultaneously computes the Hausdorff dimension along with the Fourier coefficients for both the cuspidal and the flare expansions.
We include these results for a few example values of $r$ below.

McMullen's eigenvalue algorithm is implemented in PARI \cite{PARI2} and is available on his website\footnote{\url{https://people.math.harvard.edu/~ctm/programs/index.html}} as the program \texttt{hdim}.
We note that in McMullen's code, he parameterizes infinite volume Hecke groups by the value
$$
R = 2r
$$
where $r$ is the parameter defined for us in Section \ref{ch:Hecke}.
This must be taken into account when checking outputs from \texttt{hdim} against our own results.

\begin{table}[h!]
	\centering
	\begin{tabular}{|c|c|c|c|}
	\hline
	$r$ & 0.35 & 0.40 & 0.45 \\
	\hline
	$\delta$ & 0.767052417 & 0.8169416563 & 0.8782699772 \\
	\hline
	$a_0$ & 1.0 & 1.0 & 1.0 \\
	\hline
	$a_1$ & 1.3915582132 & 1.511667648 & 1.383662851 \\
	\hline
	$a_2$ & 1.042527677 & 0.5557801545 & -0.4018405456 \\
	\hline
	$a_3$ & 0.4185419039 & -0.2633463357 & -0.1590048682 \\
	\hline
	$a_4$ & 0.0082837207 & 0.0563845298 & 1.2672198274 \\
	\hline
	$b_0$ & 0.5693837173 & 0.5509892722 & 0.5168430765 \\
	\hline
	$b_1$ & -8.6593303021E-5 & -3.8018774819E-6 & -4.2062834491E-9 \\
	\hline
	$b_2$ & 2.7659692148E-10 & -2.0592480594E-13 & -1.1328909449E-18 \\
	\hline
	$b_3$ & -2.1134025239E-14 & -2.2707863778E-18 & 2.1364421568E-26 \\
	\hline
	$b_4$ & 4.6505587029E-19 & 8.2512318578E-25 & 8.6040972875E-33 \\
	\hline
	\end{tabular}
	\caption{Hausdorff dimensions and first 5 cuspidal and flare coefficients of the base eigenfunction for various infinite volume Hecke groups. The eigenfunctions are scaled so that the $0^{th}$ cuspidal coefficient is 1.}
	\label{hecke_table}
	
\end{table}


\bibliographystyle{alpha}
\bibliography{refs.bib}

\begin{appendices}
	
	\section{Code Details}\label{ch:Appx_Code}
	
	The code stored at \url{https://github.com/alexkarlovitz/hejhal} was used to produce all of the results and figures presented in this paper.
	In this appendix, we single out a few of the files to make navigation of the GitHub repository easier.
	
	The folder ``Visuals'' contains the code for making figures.
	This was written in Python and makes use of the Python library Matplotlib.
	Of particular note is the file poincareModel.py, which defines the \texttt{Geodesic} and \texttt{FundamentalDomain} classes for the Poincar\'e upper half plane model of hyperbolic space.
	A \texttt{FundamentalDomain} object comprises of a list of \texttt{Geodesic} objects, and it implements a function to plot the fundamental domain.
	The file diskModel.py defines the same classes in the disk model.
	
	The folder ``Algorithms'' contains the implementations of our algorithm on various examples.
	This work is all written in PARI/GP, a computer algebra system for arbitrary precision computing \cite{PARI2}.
	The file cover\_Schottky.pari contains code for the symmetric Schottky groups discussed in Section \ref{ch:Schottky}.
	The function \texttt{secant\_method} implements the algorithm using the secant method technique for updating eigenvalue guesses, which we discussed in Section \ref{sec:HejhalOriginal}.
	We also have various examples of how to run this function in the file test\_cover\_Schottky.pari.
	The files cusp\_flare\_Hecke.pari and test\_cusp\_flare\_Hecke.pari do the same for the infinite volume Hecke groups that were the focus of Section \ref{ch:Hecke}.
	
	\section{Future Research Directions}\label{ch:Appx_Future}
	
	As described in this work, our extension of Hejhal's algorithm should apply to any group acting discontinuously on hyperbolic space so long as it contains a hyperbolic element.
	However, we saw in Sections \ref{ch:Schottky} and \ref{ch:Hecke} that many complications arise in putting the theory to practice.
	In this appendix, we describe two situations to which it would be interesting to apply our algorithm.
	
	\subsection{Hejhal's Algorithm in Cocompact Fundamental Domains}\label{sec:cocompact}
	
	As discussed in Section \ref{sec:IntroFlares}, our algorithm should apply to any finitely generated Zariski dense Fuchsian group for $\text{SL}(2, \mathbb{R})$, since it only requires the existence of a hyperbolic matrix.
	However, if the group is cocompact (meaning it has a compact fundamental domain), there are some major differences from the cases we looked at in this thesis.
	These differences mainly stem from the fact that the group no longer admits a flare domain.
	
	Let us briefly consider an interesting example which is discussed by Kontorovich in the context of Dirichlet domains \cite{kontorovich2011}.
	We define the group $\Gamma$ to be the set of determinant 1 matrices with entries in the ring of integers $\mathcal{O}_K = \mathbb{Z}[\sqrt{3}]$ of the number field $K = \mathbb{Q}[\sqrt{3}]$.
	A fundamental domain for $\Gamma\backslash\mathbb{H}$ is shown in Figure \ref{cc_FD}.
	\begin{figure}[h]
		\centering
		\includegraphics[width=\linewidth]{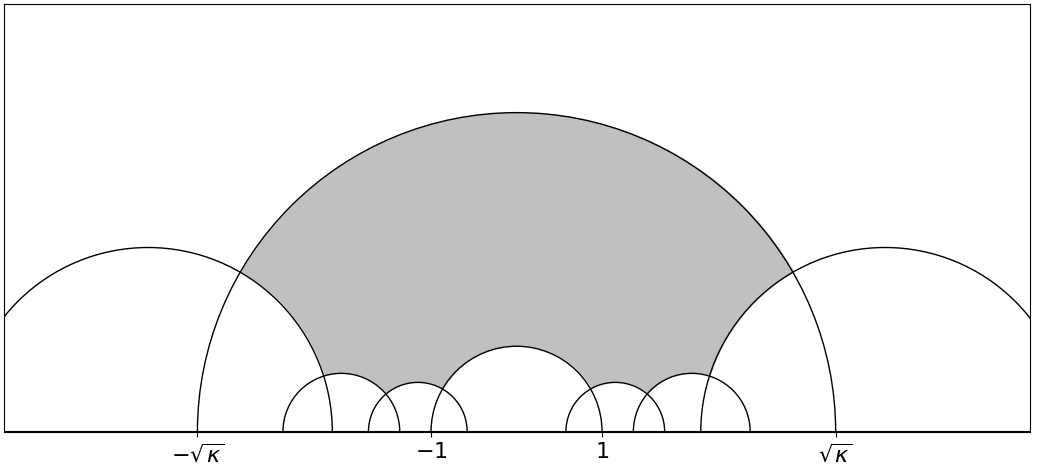}
		\caption{Fundamental domain for the action of the arithmetic group $\Gamma$ on $\mathbb{H}$; here, $\kappa = \frac{2 + \sqrt{3}}{2 - \sqrt{3}}$}
		\label{cc_FD}
	\end{figure}
	Note that $\Gamma$ contains the nontrivial diagonal matrix
	$$
	A =
	\begin{pmatrix}
		2 + \sqrt{3} & 0 \\
		0 & 2 - \sqrt{3}
	\end{pmatrix}
	$$
	Therefore, a Maass form for $\Gamma$ exhibits invariance under the map $z \mapsto \frac{2 + \sqrt{3}}{2 - \sqrt{3}}z$.
	This invariance is enough to give the flare expansion in Lemma \ref{flare_lemma}, and therefore enough to run our version of Hejhal's algorithm.
	
	The example considered here is of particular interest because of its relation to the Jacquet-Langlands correspondence.
	This correspondence relates the spectrum of the Laplacian for cocompact congruence hyperbolic surfaces to that of congruence covers of $\text{SL}(2, \mathbb{Z})\backslash\mathbb{H}$ (see \cite{bergeron2016spectrum} Chapter 8).
	For our particular example, this result suggests that the eigenvalues of Maass forms for $\Gamma$ coincide with those of the congruence subgroup $\Gamma_0(3)$.
	
	There is still some work required to apply Hejhal's algorithm to the case of cocompact Fuchsian groups.
	For example, we use the fact that flare domains have infinite volume in our proof of Lemma \ref{flare_lemma}; it is conceivable that some of the Legendre-$Q$ terms must be included in the finite volume case (particularly, we suspect, for the $n = 0$ term).
	Another consideration is if cocompactness affords any extra structure to the Fourier expansion.
	As we saw in the case of base eigenfunctions in Section \ref{sec:BaseEfuncs}, including such extra structure in Hejhal's algorithm can be key to obtaining convergence.
	
	\subsection{Hejhal's Algorithm in Hyperbolic 3-Space}\label{sec:H3}
	
	In the upper half space model of hyperbolic 3-space (which we denote $\mathbb{H}^3$), the group of orientation-preserving isometries is isomorphic to $\text{PSL}(2, \mathbb{C})$.
	The action is given by M\"obius transformations, where we interpret division as multiplication by the inverse in the quaternions:
	$$
	\begin{pmatrix}
		\alpha & \beta \\
		\gamma & \delta
	\end{pmatrix}(z) =
	(\alpha z + \beta)(\gamma z + \delta)^{-1}
	$$
	($z^{-1} = \bar{z}/|z|^2$).\footnote{see \url{http://www.maths.qmul.ac.uk/~sb/LTCCcourse/Holodyn2013notes_week3.pdf} for a good treatment of this material}
	Just as in the 2-dimensional case, we are interested in Zariski dense, geometrically finite groups $\Gamma \subset \text{PSL}(2, \mathbb{C})$ acting discretely on $\mathbb{H}^3$ with Hausdorff dimension exceeding 1 so that the Patterson-Sullivan formula $\lambda_0 = \delta(2 - \delta)$ holds.
	We also require that the group contains a hyperbolic element.
	(If $\Gamma$ contains only strictly loxodromic elements - that is, matrices with non-real trace of norm greater than 2 - then we cannot separate scaling and rotation in the Fourier expansion below).
	In this case, $\Gamma$ exhibits the analog of a flare domain in 3 dimensions.
	
	To find this analog, one starts with the formula for the Laplacian in the upper half space model
	$$
	\Delta = -y^2\left( \frac{\partial^2}{\partial x_1^2} + \frac{\partial^2}{\partial x_2^2} + \frac{\partial^2}{\partial y^2} \right) + y\frac{\partial}{\partial y}
	$$
	Then, after conjugation, we can assume the group admits invariance under the map $z \mapsto \kappa z$.
	This allows one to find a Fourier expansion for Maass forms in the same spirit as that given in Lemma \ref{flare_lemma}.
	\begin{lem}
		Let $\Gamma \subset \text{PSL}(2, \mathbb{C})$ be a group acting discontinuously on $\mathbb{H}^3$ which contains the matrix
		$\begin{psmallmatrix}
			\sqrt{\kappa} & 0 \\
			0 & \sqrt{\kappa}^{-1}
		\end{psmallmatrix}$ for some $\kappa > 1$.
		Let $f(r, \theta, \phi)$ be a Maass form for $\Gamma$ written in spherical coordinates on the upper half space model.
		Then $f$ has a Fourier expansion of the form
		\begin{equation}\label{H3_fourier}
			f(r, \theta, \varphi) = \sum_{m, n \in \mathbb{Z}}g_{m, n}(\varphi)e^{im\theta}e\left(n\frac{\log r}{\log\kappa}\right)
		\end{equation}
		where for each pair $(m, n)$, the function $g_{m, n}(\varphi)$ satisfies a hypergeometric differential equation.
	\end{lem}
	Finding the specific form of the functions $g_{m, n}$ would be the first step toward extending our algorithm to 3 dimensions.
	
	\section{Extended Results for Schottky Groups}\label{ch:Appx_Results}
	
	Here, we continue Table \ref{schottky_table} from Section \ref{sec:resultsSchottky}.
	
	\begin{longtable}{|c|c|c|c|}
		\hline
		$\theta$ & 94 & 95 & 96 \\
		\hline
		$\delta$ & 0.5063972405 & 0.5155835572 & 0.5250520005 \\
		\hline
		$b_0$ & 1.0 & 1.0 & 1.0 \\
		\hline
		$b_1$ & -2.5166508538E-7 & -1.7300804172E-7 & -1.1617826121E-7 \\
		\hline
		$b_2$ & 1.1198755812E-14 & 4.8519808552E-15 & 1.9929273692E-15 \\
		\hline
		$b_3$ & -2.0726623244E-22 & -4.2399625678E-23 & -6.1816973015E-24 \\
		\hline
		$b_4$ & -2.3468239788E-29 & -5.5134294682E-30 & -1.141094769E-30 \\
		\hline
		$b_5$ & 3.4777754033E-36 & 4.8036304002E-37 & 5.505497463E-38 \\
		\hline
		$b_6$ & -2.7250500511E-43 & -2.7978189536E-44 & -1.7376105471E-45 \\
		\hline
		$b_7$ & -6.1842978186E-51 & -7.4067873863E-51 & 5.4124177087E-53 \\
		\hline
		$b_8$ & -2.3465059548E-57 & -4.0643898947E-57 & 2.1949093323E-59 \\
		\hline
		$b_9$ & 4.5291057657E-64 & 1.7973473077E-64 & 9.3505260894E-66 \\
		\hline
		~ & ~ & ~ & ~ \\
		\hline
		$\theta$ & 97 & 98 & 99 \\
		\hline
		$\delta$ & 0.5348189358 & 0.5449022229 & 0.5553214236 \\
		\hline
		$b_0$ & 1.0 & 1.0 & 1.0 \\
		\hline
		$b_1$ & -7.6013378206E-8 & -4.8314592354E-8 & -2.9729974038E-8 \\
		\hline
		$b_2$ & 7.7120365664E-16 & 2.7910594923E-16 & 9.3644442708E-17 \\
		\hline
		$b_3$ & -1.1629968189E-25 & 3.240903982E-25 & 1.4463862455E-25 \\
		\hline
		$b_4$ & -2.0560710306E-31 & -3.2250222589E-32 & -4.3389494219E-33 \\
		\hline
		$b_5$ & 5.6269123662E-39 & 4.831982001E-40 & 3.4232144892E-41 \\
		\hline
		$b_6$ & -8.8541082608E-47 & -4.1736843832E-48 & -2.2467648948E-49 \\
		\hline
		$b_7$ & 2.7068359385E-54 & 6.7122586307E-56 & -3.0306326758E-58 \\
		\hline
		$b_8$ & 1.0664361431E-61 & 5.4886600578E-63 & 1.5496208223E-64 \\
		\hline
		$b_9$ & 2.4678823261E-68 & 4.9079480835E-70 & -7.6094629613E-72 \\
		\hline
		~ & ~ & ~ & ~ \\
		\hline
		$\theta$ & 100 & 101 & 102 \\
		\hline
		$\delta$ & 0.5660980508 & 0.5772558693 & 0.5888212627 \\
		\hline
		$b_0$ & 1.0 & 1.0 & 1.0 \\
		\hline
		$b_1$ & -1.7639423974E-8 & -1.0043266514E-8 & -5.456246873E-9 \\
		\hline
		$b_2$ & 2.8821033506E-17 & 8.0301043236E-18 & 1.9919927194E-18 \\
		\hline
		$b_3$ & 4.1975548533E-26 & 9.4718196155E-27 & 1.7133977096E-27 \\
		\hline
		$b_4$ & -4.9802480099E-34 & -4.6307356889E-35 & -3.7448841415E-36 \\
		\hline
		$b_5$ & 1.849630679E-42 & 7.6430925784E-44 & 2.5955783828E-45 \\
		\hline
		$b_6$ & -2.5041130675E-51 & -1.3916147359E-52 & -1.3656888855E-53 \\
		\hline
		$b_7$ & 4.0025314834E-58 & 3.4046971526E-59 & -4.8792676259E-61 \\
		\hline
		$b_8$ & -1.8672502894E-66 & -8.1716268477E-67 & -8.5441042622E-69 \\
		\hline
		$b_9$ & 3.574917711E-73 & 2.5477503938E-74 & -6.6213464918E-77 \\
		\hline
		~ & ~ & ~ & ~ \\
		\hline
		$\theta$ & 103 & 104 & 105 \\
		\hline
		$\delta$ & 0.6008236865 & 0.6132962338 & 0.6262763513 \\
		\hline
		$b_0$ & 1.0 & 1.0 & 1.0 \\
		\hline
		$b_1$ & -2.8091147305E-9 & -1.3592086394E-9 & -6.117866583E-10 \\
		\hline
		$b_2$ & 4.3059867697E-19 & 7.8634405884E-20 & 1.1575941563E-20 \\
		\hline
		$b_3$ & 2.5470168364E-28 & 3.0122225886E-29 & 2.7838285274E-30 \\
		\hline
		$b_4$ & -2.1966625504E-37 & -1.0172542862E-38 & -3.4292418033E-40 \\
		\hline
		$b_5$ & 8.4731393595E-47 & 1.5448429246E-48 & 2.2652082015E-50 \\
		\hline
		$b_6$ & -1.6101058727E-56 & -9.03295402E-59 & -1.240719692E-60 \\
		\hline
		$b_7$ & -6.9814011287E-67 & -4.1657738196E-69 & -5.1850882962E-69 \\
		\hline
		$b_8$ & -3.983344528E-76 & 8.7894985482E-78 & 1.219232274E-79 \\
		\hline
		$b_9$ & 4.8710709698E-84 & -1.1631482431E-86 & 7.2144594008E-87 \\
		\hline
		~ & ~ & ~ & ~ \\
		\hline
		$\theta$ & 106 & 107 & 108 \\
		\hline
		$\delta$ & 0.6398067604 & 0.6539366615 & 0.6687233406 \\
		\hline
		$b_0$ & 1.0 & 1.0 & 1.0 \\
		\hline
		$b_1$ & -2.5291540896E-10 & -9.4495240438E-11 & -3.1254026951E-11 \\
		\hline
		$b_2$ & 1.2579941054E-21 & 7.9104214829E-23 & -1.4847943528E-24 \\
		\hline
		$b_3$ & 1.9451181016E-31 & 9.798116181E-33 & 3.3450030796E-34 \\
		\hline
		$b_4$ & -6.1528001369E-42 & -1.2039748589E-43 & -7.2517362158E-46 \\
		\hline
		$b_5$ & 4.6733173301E-52 & 5.1108190527E-55 & 3.0778195335E-56 \\
		\hline
		$b_6$ & -1.2962178464E-61 & 2.582536327E-66 & -7.6195707277E-67 \\
		\hline
		$b_7$ & 2.6552201906E-70 & 1.1956725704E-76 & 3.5572546476E-76 \\
		\hline
		$b_8$ & 4.9493375274E-79 & 9.1018233297E-87 & 5.4814226659E-87 \\
		\hline
		$b_9$ & 9.3923926854E-89 & 1.3108018334E-96 & -4.2411477229E-96 \\
		\hline
		~ & ~ & ~ & ~ \\
		\hline
		$\theta$ & 109 & 110 & 111 \\
		\hline
		$\delta$ & 0.6842343612 & 0.7005506322 & 0.7177708377 \\
		\hline
		$b_0$ & 1.0 & 1.0 & 1.0 \\
		\hline
		$b_1$ & -8.905490517E-12 & -2.1078083478E-12 & -3.9408565331E-13 \\
		\hline
		$b_2$ & -9.0798851393E-25 & -9.2695203422E-26 & -4.6177556013E-27 \\
		\hline
		$b_3$ & 7.1254493292E-36 & 8.4731543192E-38 & 4.7860872619E-40 \\
		\hline
		$b_4$ & -4.8360048672E-48 & -9.6547812524E-51 & -1.0054790036E-53 \\
		\hline
		$b_5$ & -3.5921262497E-61 & -3.254837943E-64 & -2.0571402777E-66 \\
		\hline
		$b_6$ & -8.4531009525E-72 & 1.9270339967E-76 & 2.1383139366E-78 \\
		\hline
		$b_7$ & -8.835095445E-82 & 4.1506421257E-89 & -7.5887908411E-91 \\
		\hline
		$b_8$ & -2.8351626532E-92 & 1.0764680261E-98 & 4.311246004E-102 \\
		\hline
		$b_9$ & -1.1017111646E-102 & 1.8841276344E-110 & -1.0579301794E-116 \\
		\hline
		~ & ~ & ~ & ~ \\
		\hline
		$\theta$ & 112 & 113 & 114 \\
		\hline
		$\delta$ & 0.7360180662 & 0.7554501793 & 0.7762769358 \\
		\hline
		$b_0$ & 1.0 & 1.0 & 1.0 \\
		\hline
		$b_1$ & -5.4178437433E-14 & -4.9247119445E-15 & -2.5071014886E-16 \\
		\hline
		$b_2$ & -1.1159016934E-28 & -1.1069761375E-30 & -3.3016384588E-33 \\
		\hline
		$b_3$ & 1.0378217055E-42 & 6.2801367977E-46 & 6.3119677388E-50 \\
		\hline
		$b_4$ & -2.1495439013E-57 & -1.4686177078E-62 & 2.260854543E-66 \\
		\hline
		$b_5$ & 1.0245177302E-70 & -3.0934984328E-76 & 1.7310619203E-81 \\
		\hline
		$b_6$ & 9.1784411082E-83 & 1.8002141971E-89 & 3.8122640372E-96 \\
		\hline
		$b_7$ & 4.9151604494E-96 & 2.19079449E-103 & 3.8220683218E-111 \\
		\hline
		$b_8$ & 2.4417894817E-109 & 1.3494436105E-117 & 1.6223366286E-126 \\
		\hline
		$b_9$ & -2.2789932004E-121 & 6.144589999E-131 & 2.3495874332E-141 \\
		\hline
		~ & ~ & ~ & ~ \\
		\hline
		$\theta$ & 115 & 116 & 117 \\
		\hline
		$\delta$ & 0.7987903036 & 0.823423288 & 0.8508798156 \\
		\hline
		$b_0$ & 1.0 & 1.0 & 1.0 \\
		\hline
		$b_1$ & -5.418283051E-18 & -2.9946987419E-20 & -1.4682715922E-23 \\
		\hline
		$b_2$ & -1.7162852646E-36 & -5.7205412456E-41 & -1.4617186967E-47 \\
		\hline
		$b_3$ & 6.2989019301E-55 & 3.2457831213E-62 & 7.5982009863E-74 \\
		\hline
		$b_4$ & 3.0968327766E-72 & 3.3998460484E-83 & 2.651059256E-96 \\
		\hline
		$b_5$ & 9.5710565755E-89 & -1.1995692265E-101 & -1.262557898E-118 \\
		\hline
		$b_6$ & 6.0648983058E-105 & -3.830712769E-120 & 1.2426594636E-140 \\
		\hline
		$b_7$ & 1.119509805E-121 & -1.0033974478E-138 & -1.1253026592E-162 \\
		\hline
		$b_8$ & 3.9841535568E-138 & -2.6518406123E-157 & 9.5751944735E-185 \\
		\hline
		$b_9$ & 1.2651190276E-154 & -5.635630824E-176 & -7.540633254E-207 \\
		\hline
		~ & ~ & ~ & ~ \\
		\hline
		$\theta$ & 118 & 119 & 120 \\
		\hline
		$\delta$ & 0.8824840727 & 0.9215247961 & 1.0 \\
		\hline
		$b_0$ & 1.0 & 1.0 & 1.0 \\
		\hline
		$b_1$ & -4.1125246082E-29 & -1.2136119956E-41 & 0.0 \\
		\hline
		$b_2$ & -1.1900011198E-58 & -1.0415524924E-83 & 0.0 \\
		\hline
		$b_3$ & -6.794358068E-89 & -3.0935852244E-126 & 0.0 \\
		\hline
		$b_4$ & -5.3083424343E-119 & -1.5403528265E-168 & 0.0 \\
		\hline
		$b_5$ & 3.110564025E-148 & 2.2937221356E-208 & 0.0 \\
		\hline
		$b_6$ & -1.2034894694E-176 & 1.551286008E-248 & 0.0 \\
		\hline
		$b_7$ & 8.2458855624E-205 & -1.3164901303E-289 & 0.0 \\
		\hline
		$b_8$ & 3.9051264579E-233 & -5.9952485673E-329 & 0.0 \\
		\hline
		$b_9$ & -4.6050893399E-261 & -4.3814548545E-369 & 0.0 \\
		\hline
		
		\caption{Hausdorff dimensions and first 10 Fourier coefficients of the base eigenfunction for various symmetric Schottky groups. The eigenfunctions are normalized so that the $0^{th}$ coefficient is 1. Note that for $\theta = 120$ degrees, the group is finite volume; this implies that constant functions are $L^2$, and they have eigenvalue $0$.}
		\label{schottky_table_long}
		
	\end{longtable}
	
\end{appendices}

\end{document}